\documentclass[english,reqno]{amsart}
\usepackage[margin=1in]{geometry}
\usepackage{amssymb,amsbsy,amsmath,amsfonts,amssymb,amscd}
\usepackage{comment}
\usepackage{hyperref}
\usepackage{xcolor}
\usepackage{algorithm,algpseudocode}
\newtheorem{theorem}{Theorem}[section]
\newtheorem{rmk}{Remark}[section]

\newtheorem{proposition}{Proposition}[section]
\newtheorem{lem}{Lemma}[section]
\newtheorem{definition}{Definition}

\newcommand{\pos}{\mathrm{pos}}
\newcommand{\Cov}{\mathrm{Cov}}

\newcommand{\MCG}{\mathcal{G}}
\newcommand{\te}{\textbf{e}}

\newcommand{\tu}{\textbf{u}}

\newcommand{\EE}{\mathbb{E}}
\newcommand{\RR}{\mathbb{R}}
\newcommand{\Rm}{\mathrm{m}}

\title{Ensemble Kalman Sampler: mean-field limit and convergence analysis}

\author{Zhiyan Ding}
\address{Mathematics Department, University of Wisconsin-Madison, 480 Lincoln Dr., Madison, WI 53705 USA.}
\email{zding49@math.wisc.edu}
\author{Qin Li} 
\address{Mathematics Department and Wisconsin Institutes of Discoveries, University of Wisconsin-Madison, 480 Lincoln Dr., Madison, WI 53705 USA.}
\email{qinli@math.wisc.edu}

\date{\today}

\thanks{The research of Z.D. and Q.L. was supported in part by National Science Foundation under award 1619778, 1750488 and Wisconsin Data Science Initiative. Both authors would like to thank Andrew Stuart for the helpful discussions.}

\begin{document}
\begin{abstract}
Ensemble Kalman Sampler (EKS) is a method introduced in~\cite{EnFL} to find approximately $i.i.d.$ samples from a target distribution. As of today, why the algorithm works and how it converges is mostly unknown. The continuous version of the algorithm is a set of coupled stochastic differential equations (SDEs). In this paper, we prove the wellposedness of the SDE system, justify its mean-field limit is a Fokker-Planck equation, whose long time equilibrium is the target distribution. We further demonstrate that the convergence rate is near-optimal ($J^{-1/2}$, with $J$ being the number of particles). These results, combined with the in-time convergence of the Fokker-Planck equation to its equilibrium~\cite{Carrillo-2003}, justify the validity of EKS, and provide the convergence rate as a sampling method.
\end{abstract}

\maketitle

\section{Introduction}
Sampling from a target distribution is a core problem in Bayesian statistics, machine learning and data assimilation. It has wide applications in atmospheric science, petroleum engineering, remote sensing and epidemiology in the form of volume computation, and bandit optimization~\cite{FABIAN198117,PES,Convexproblem,ATTS}.

A large number of sampling methods have been proposed, and many have shown to be successful under certain circumstances. This includes the traditional methods such as Markov chain Monte Carlo  (MCMC)~\cite{Robert2004,roberts1996}, Langevin dynamics based methods (including both the overdamped Langevin Monte Carlo~\cite{PARISI1981378,roberts1996,doi:10.1111/rssb.12183,DALALYAN20195278} and underdamped Langevin Monte Carlo~\cite{10.5555/3044805.3045080,10.5555/2969442.2969566,Cheng2017UnderdampedLM,eberle2019}) and the newly developed Stein Variational Gradient Descent (SVGD)~\cite{NIPS_Liu_SVGD}, Hamiltonian Monte Carlo methods~\cite{HMC,mangoubi2017rapid,10.5555/3327345.3327502}, and their different levels of combination (such as MALA)~\cite{roberts1996,LDMHA2002,dwivedi2018logconcave,MALA3}. Weighted particles are also considered, and this leads to importance sampling~\cite{IM1989,Doucet2001,SMCBOOK} and the birth-death Langevin sampler \cite{lu2019accelerating}. In the recent years, there has been a boost in designing and analyzing ensemble methods. This means a large number of particles are sampled first according to an easy-to-sample distribution (such as Gaussian or uniform), and moved around according to certain dynamics, hoping in finite time, they reconstruct the target distribution. Some famous methods in this direction include the Ensemble Kalman Inversion (EKI)~\cite{DAEnKF,Iglesias_2013}(derived from Ensemble Kalman filter~\cite{Evensen2003,DAEnKF}) and Kalman-Bucy filter~\cite{Bergemann_Reich10_mollifier,Bergemann_Reich10_local,deWiljes}). They have attracted a large amount of numerical and analytical studies~\cite{SS,SS2,DCPS,ding2019meanfield}. Also see a very insightful review~\cite{Reich2011}.

In~\cite{EnFL}, the authors, inspired by the idea from PDE gradient flow and the ensemble structure of EKI, proposed a new method, termed Ensemble Kalman Sampler (EKS). The method works rather well in computation and the intuition is clear. But to today, the rigorous theoretical justification is mostly unknown. In particular, the method runs $J$ number of particles for a certain amount of time $T$, but due to the lack of error analysis, we do now know to how to set these parameters to achieve a preset accuracy.

The main goal of the current paper is to give a sharp error estimate of the algorithm EKS in the linear setting. To do so, we first characterize the continuous version of the algorithm using a coupled SDE system. From there, we use the following tools: 1. the Lyapunov theory for showing the wellposedness of the SDE system; 2. the mean-field limit argument to transfer the analysis of the SDE system to a Fokker-Planck PDE; 3. the convergence analysis of Fokker-Planck equation. The second tool explains the convergence in $J$ and the third tool explains the convergence in $T$.

We emphasize that in~\cite{EnFL} the authors have already proved the exponential convergence in $T$, and our contribution is mainly in 1, showing the wellposedness, and 2, showing the mean-field limit. The wellposedness of the SDE is an interesting topic by itself, and it also serves as a crucial component in the carrying out the mean-field limit, deeming this part of analysis necessary. We also emphasize that it is not yet our interest to compare different sampling methods in this article. We would rather focus on one particular method (EKS) and give a sharp error estimate. Hopefully this serves as a building block for future investigations in comparing methods.

We also emphasize that it is a simple fact that EKS is \emph{not} a consistent method in the nonlinear setting, in the sense that it does not produce the target distribution for whatever $J$ and $T$. We confine our mean-field limit discussion to the linear setting merely because the method is not correct in the nonlinear setup, and the argument is rather trivial, rendering the proof for its mean-field limit pointless. We nevertheless discuss how wrong the method can be in Appendix~\ref{sec:appendixnonlinear}. In particular, we write down the SDE system that provides the correct convergence and compare it with what EKS uses. As will be stated later, that the SDE system EKS relies on is a mere finite difference approximation to the true SDE system, and thus could not be right when the finite difference approximation breaks down in the nonlinear setting. However, computing this correct SDE system does not seem to be numerically easy, especially because it loses the gradient-free property. As a consequence, we do not pursue it either practically (designing an algorithm), or theoretically (showing its mean-field limit rigorously). We do give the wellposedness proof for the correct SDE system in Appendix~\ref{sec:appendixnonlinear}.

Showing the wellposedness using the Lyapunov theory is a standard practice. Upon which we also obtain the boundedness of high moments. Proving the mean-field limit, however, becomes significantly more difficult for the particular set of SDEs we are investigating that arises from EKS. Indeed, there are many models whose mean-field limits have been rigorously established in literature, but we find the techniques are not entirely adaptable to our situation. As will be presented in Section~\ref{sec:mean_field}, to show the mean-field limit, we mainly adopt coupling method~\cite{Sznitman}, by first represent the PDE with its intrinsic SDE system, and then compare the two SDE systems (the one arises from the limiting PDE, and the original one derived from EKS).

 Different systems have different regularity of the transport and Brownian motion coefficients, leading to different levels of technical difficulties. The most straightforward situation is when the Brownian motion coefficient is a constant, and the transport coefficient satisfies some kind of Lipschitz condition. In this situation, by subtracting the two SDE systems, the Brownian motion terms vanish, and the transport term is bounded directly by the disparity of the SDEs, prompting the use of the Gr\"onwall inequality for the error bound. This situation is seen in~\cite{braun1977}. If the transport term is nonlinear, or even a functional of the SDE itself, as is the case in most practical settings, the Lipschitz condition is hard to obtain. One then manually draws a large domain to have the Lipschitz condition to hold true inside the domain, and compensate the rest of the proof by demonstrating that the probability for the particles to go outside the domain is small~\cite{Bolley_Carrillo,huang2018meanfield,Lazarovici2015AMF,jin2018random}. In particular, in~\cite{Bolley_Carrillo}, one bounds the exponential moment of particle ($\EE(e^{|u|})<\infty$), and in \cite{huang2018meanfield,Lazarovici2015AMF}, the singularities of the interaction kernel induced by the Poisson equation gives the guidance for the domain cut. However, for these methods to be used, it is crucial to have the coefficient for the Brownian motion being constants, so that when one compares the two systems, the Brownian motion effect vanishes. In the case when Brownian motion coefficient is non-constant, as seen in the Mckean-Vlasov case~\cite{Sznitman,Mlard1996} and neuron models~\cite{Boss-2014}, to the best of our knowledge, some kind of Lipschitz condition is used.

The SDE we encounter is different from the ones in the previous studies: it has functional coefficients for both the transport and the Brownian motion terms. This makes most available previous approaches not applicable to our setting. To overcome the difficulty, we employ a bootstrapping argument on $L^2$ norm. To a large extent, we first assume the error decays with certain rate in $J$, and show that such decay rate can be tightened, till we reach the threshold $-1/2+\epsilon$. In this tightening process, we apply the H\"older inequality, and move some of the weights to $L^q$ with $q>2$. To bound these terms, one needs the control of the high moments of the SDE and show it is independent of $J$. This part of the preparation work is done in the section \ref{sec:SDE} where we show the wellposedness for the SDE system. 

The rest of the paper is organized as the following. In Section~\ref{sec:method} we describe the algorithm and present its continuum limit. It is a set of coupled SDEs, and is the model that we will study. In Section~\ref{sec:main_results} we present the main results along with the roadmap of the proof. We divide the proof into three steps, and we summarize the results in each step. Technical proofs are collected in Section~\ref{sec:SDE}-\ref{sec:mean_field} and the flowchart of the relation between lemmas and propositions is presented at the end of Section~\ref{sec:main_results} after we present the roadmap.

\section{Ensemble Kalman Sampler and the continuum limit}\label{sec:method}
Ensemble Kalman Sampler (EKS) is an algorithm proposed in~\cite{EnFL} to find approximately i.i.d. samples from a target distribution. It is a core problem in Bayesian inverse problem and machine learning, in which the target distribution is usually the posterior distribution formulated through an inverse problem setup.

In inverse problems, measurements are taken to infer the unknown parameters in the physical system. Let $u \in \mathbb{R}^L$ be the to-be-reconstructed parameter and $y \in \mathbb{R}^K$ be the measurements, then a typically setup is to denote $\mathcal{G}$, a forward map, or the parameter-to-observable map, that maps $u$ to $y$:
\[
y = \mathcal{G}(u) + \eta\,.
\]
Here $\eta$ denotes the noise in the measurement-taking. While the forward problem amounts to finding $y$ for any given $u$, the inverse problem amounts to reconstructing $u$ from $y$ with some assumed knowledge on $\eta$. A typical assumption is to set $\eta\sim \mathcal{N}(0, \Gamma)$, a Gaussian noise independent of $u$, then the loss functional $\Phi(\cdot;y):\mathbb{R}^K\rightarrow\mathbb{R}$ becomes:
\[
\Phi(u;y)=\frac{1}{2}\left|y-\mathcal{G}(u)\right|^2_\Gamma\,,\quad\text{where}\quad\left|\ \cdot\ \right|_\Gamma:=\left|\Gamma^{-\frac{1}{2}}\ \cdot\ \right|\,.
\]
The Bayes' theorem states that the posterior density is the (normalized) product of the prior density and the likelihood function:
\begin{equation*}
\rho_\pos(u)=\frac{1}{Z}\exp{\left(-\Phi(u;y)\right)}\rho_0(u)\,,\quad\text{with}\quad Z:=\int_{\mathbb{R}^L}\exp\left(-\Phi(u;y)\right)\rho_0(u)du\,.
\end{equation*}
Here $Z$ serves as the normalization factor, $\exp\left(-\Phi(u;y)\right)$ serves as the likelihood function and $\rho_0$ serves as the prior density function that collects people's prior knowledge about the distribution of $u$. This posterior distribution represents the probability measure of the to-be-reconstructed parameter $u$, blending the prior knowledge and the collected data $y$, taking $\eta$, the measurement error into account. More details on Bayesian inversion can be found in \cite{Dashti2017,stuart_2010}.

When the prior distribution is a Gaussian, and the forward map is linear, the posterior distribution can be explicitly written down. Suppose the prior distribution is a Gaussian distribution with mean $u_0$ and covariance $\Gamma_0$:
\begin{equation}\label{priordistribution}
\rho_\text{prior}(u)\propto\exp\left(-\frac{1}{2}\left(u-u_0\right)^\top \Gamma^{-1}_0\left(u-u_0\right)\right)\,,
\end{equation}
and that $\mathcal{G}$ is linear, meaning: there exists a matrix $A$ so that
\begin{equation}\label{linear}
\mathcal{G}(\cdot)=A\cdot\,,\quad\text{with}\quad A\in\mathcal{L}(\mathbb{R}^L,\mathbb{R}^K)\,,
\end{equation}
then the cost function is:
\begin{equation}\label{PhiRdef}
\Phi_R(u;y)=\frac{1}{2}|y-Au|^2_{\Gamma}+\frac{1}{2}\left|u-u_0\right|^2_{\Gamma_0}\,.
\end{equation}
and the covariance and the mean of the posterior distribution are
\begin{equation}\label{def:Bustar}
\mathrm{Cov}^{-1}_{\rho_\text{pos}}=B=A^\top\Gamma^{-1}A+\Gamma^{-1}_0\,,\quad \EE_{\rho_\text{pos}}=u^\ast=B^{-1}\left(A^\top\Gamma^{-1}y+\Gamma^{-1}_0u_0\right)\,,
\end{equation}
which makes
\[
\rho_{\text{pos}}(u)\propto\exp\left(-\frac{1}{2}\left|u-u^\ast\right|^2_{B^{-1}}\right)\,.
\]

\subsection{Algorithm description}
The EKS is an algorithm for finding approximately i.i.d. samples for the target distribution $\rho_\pos$. Unlike the traditional methods such as MCMC and LMC in which particles are sequentially proposed, in ensemble type sampling methods, a large number of particles are drawn from potentially arbitrary distribution at the initial time, and are moved around by some kind of actions along time evolution. After certain time, it looks like the particles are drawn from the target distribution. Different ensemble methods use different strategies to introduce these physical actions. EKI, for example, introduces a linear line that connects the prior and the posterior distribution in the function space on the log scale, and in EKS, the authors design a gradient flow on the function space that drives any given function (with certain regularity) to the target one.

In theory, if the gradient flow is followed exactly, the target distribution can be found perfectly. However, the coefficients in the gradient flow depends on the underlying solution itself, which is not available numerically. So numerically one replaces it by its ensemble version, hoping such replacement does not cause too much error. Showing the mean-field limit essentially comes down to justifying that this error brought by the replacement is indeed small.

The method is summarized in Algorithm \ref{ALG1}.

\begin{algorithm}[h]
\caption{\textbf{Ensemble Kalman sampler}}\label{ALG1}
\begin{algorithmic}
\State \textbf{Preparation:}

\State 1. Input: $J$ (number of particles); $h$ (stepsize); $N$ (stopping index); $\Gamma$; $\Gamma_0$; and $y$ (data).
\State 2. Initial: $\{u^j_0\}$ sampled from a initial distribution induced by a density function $\rho_0$.

\State \textbf{Run: } Set time step $n=0$;
\State \textbf{While} $n<N$:

1. Define empirical means and covariance:
\begin{align}\label{eqn:en_mean_var}
\overline{u}_n=\frac{1}{J}\sum^J_{j=1}u^j_n\,,\quad&\text{and}\quad\overline{\MCG}_n=\frac{1}{J}\sum^J_{j=1}\MCG(u^j_n)\,,\nonumber\\
\Cov_{u_n,u_n}=\frac{1}{J}\sum^J_{j=1}\left(u^j_n-\overline{u}_n\right)\otimes \left(u^j_n-\overline{u}_n\right)\,,\quad &\text{and}\quad
\Cov_{u_n,\mathcal{G}_n}=\frac{1}{J}\sum^J_{j=1}\left(u^j_n-\overline{u}_n\right)\otimes \left(\MCG(u^j_n)-\overline{\MCG}_n\right)\,.
\end{align}

2. Update ensemble particles ($\forall 1\leq j\leq J$)
\begin{equation}\label{eqn:update_ujn}
\begin{aligned}
&u^j_{*,n+1}=u^j_n-h\Cov_{u_n,\mathcal{G}_n}\Gamma^{-1}\left(\MCG(u^j_n)-y\right)-h\Cov_{u_n,u_n}\Gamma^{-1}_0\left(u^j_{*,n+1}-u_0\right)\,,\\
&u^j_{n+1}=u^j_{*,n+1}+\sqrt{2h\Cov_{u_n,u_n}}\xi^j_n\,,\quad\text{with}\quad\xi^j_{n+1}\sim \mathcal{N}(0,\mathrm{I})\,.
\end{aligned}
\end{equation}

3. Set $n\to n+1$.
\State \textbf{end}
\State \textbf{Output:} Ensemble particles $\{u^j_N\}$.
\end{algorithmic}
\end{algorithm}

There are a few parameters in the algorithm:
\begin{itemize}
\item[1.] $T=Nh$ is the stopping time, with $h$ being the stepsize, and $N$ being the number of iterations. The hope is to show the convergence to the target distribution is exponentially fast in $T$.
\item[2.] $J$ is the number of particles fixed ahead of time. The hope is to show that when $J\gg 1$, the ensemble distribution of the particles converges to the target distribution at the order of $1/\sqrt{J}$ for any finite $T$. This is the optimal rate one can hope for in the framework of Monte Carlo.
\item[3.] $\rho_0$ is the initial density function. It is not necessarily required that $\rho_0$ being equivalent to $\rho_\text{prior}$. As will be shown in the later sections, the mean-field limit argument holds true as long as $\rho_0$ is smooth and have bounded high moments.
\end{itemize}

\subsection{Continuum limit of Ensemble Kalman sampler}
The algorithm is discrete in time. As $h\to 0$, one achieves its continuum limit. In particular, setting $h\rightarrow0$ in~\eqref{eqn:update_ujn}, one has, for all $j$:
\begin{equation}\label{Ecov_dis_revise}
du^{j}_t=-\mathrm{Cov}_{u_t,\mathcal{G}_t}\Gamma^{-1}(\mathcal{G}(u^j_t)-y)dt-\mathrm{Cov}_{u_t,u_t}\Gamma^{-1}_0(u^j_t-u_0)dt+\sqrt{2\mathrm{Cov}_{u_t,u_t}}dW^{j}_t\,,
\end{equation}
where $\mathrm{Cov}_{u_t,\mathcal{G}_t}$, $\mathrm{Cov}_{u_t,u_t}$ are empirical variances similarly defined as in~\eqref{eqn:en_mean_var}.

In the linear setting, assuming~\eqref{linear}, with the $\Phi_R$ definition in~\eqref{PhiRdef}, equation~\eqref{Ecov_dis_revise} can be written as:
\begin{equation}\label{CAmunew}
du^{j}_t=-\mathrm{Cov}_{u_t,u_t}\nabla\Phi_R(u^j_t)dt+\sqrt{2\mathrm{Cov}_{u_t,u_t}}dW^{j}_t\,.
\end{equation}

We further define $M_{u_t}(du)$ to be the ensemble distribution:
\begin{equation}\label{eqn:empirical}
M_{u_t}=\frac{1}{J}\sum^{J}_{j=1}\delta_{u^{j}_t}\,.
\end{equation}
The goal of this paper is to give a quantitative estimate of how this empirical distribution, with the particles guided in~\eqref{CAmunew} converges to the target distribution in both time $T$ and the number of particles $J$, in Wasserstein distance.

\begin{rmk}
Some remarks are in order:
\begin{itemize}
\item It has been a tradition to design sampling method that converges as $J\to\infty$, namely as $J\to\infty$ in long time the ensemble distribution becomes the invariant measure (the target distribution). In a five-page small note~\cite{nusken2019note}, the authors provide a very insightful adjustment to the ``flux" term so that the invariant measure can be achieved by any finite number of samples in long time as well.
\item The thorough numerical analysis should also include $h$ dependence. Namely, one should prove $M_{u_t}$ converges to $\rho_\pos$ when $h\to 0$, $T\to\infty$ and $J\to\infty$. The $h\to0$ amounts to give a rigorous justification of the Euler-Maruyama method for the SDE~\eqref{CAmunew}. We regard this part of the work detached from the current setting, both in terms of the goal, and in them of the required technicality, and we do not pursue the direction.
\item In the original paper~\cite{EnFL} the authors arrived at~\eqref{Ecov_dis_revise} using the approximation
\[
\mathrm{Cov}_{u_t,\mathcal{G}_t}\Gamma^{-1}(\mathcal{G}(u_t)-y)\approx \mathrm{Cov}_{u_t}\nabla_u\left[\frac{1}{2}\left|\mathcal{G}(u_t)-y\right|^2_\Gamma\right]\,.
\]
This approximation holds true only if $\mathcal{G}$ is completely linear. This suggests the algorithm is valid only in the linear setting. This is the main reason for us to confine our mean-field investigation to the linear setup. Potentially when $\mathcal{G}$ is close to a linear function, the approximation would still be rather accurate. In Appendix~\ref{sec:appendixnonlinear} we study how wrong it would be if $\mathcal{G}$ is nonlinear. We should also emphasize, in the study of data assimilation (instead of sampling) methods, dynamics is introduced to add time dependence to the parameter $u$, and one uses both the data $y$ and the underlying dynamics to configure $u$. In those setups, nonlinear dynamics was indeed discussed, as seen in~\cite{Reich2011,deWiljes}. However, in both cases, $y$'s dependence on $u$ is still linear. We have not seen analytical studies of ensemble methods that investigate sampling of $u$ using $y$ information when the relation is nonlinear except that in~\cite{ding2019meanfield}.
\end{itemize}
\end{rmk}

\section{Main results and strategy of our proof}\label{sec:main_results}
We now present our main results and the roadmap of proof in this section. In the end the goals can be split into the following three sub-goals:
\begin{itemize}
\item[No. 1:] The SDE system~\eqref{CAmunew} derived directly from the algorithm EKS is a wellposed system. The precise statement of this result is in Theorem~\ref{ThforM}.
\item[No. 2:] The mean-field limit of the SDE system~\eqref{CAmunew} is the Fokker-Planck equation:
\begin{equation}\label{meanfieldstep2}
\partial_t\rho=\nabla\cdot(\rho\Cov_{\rho(t),\mathcal{G}}\Gamma^{-1}(\mathcal{G}(u)-y))+\nabla\cdot(\rho\Cov_{\rho(t)}\Gamma^{-1}_0(u-u_0))+\mathrm{Tr}\left(\Cov_{\rho(t)}D^2\rho\right)\,.	
\end{equation}
When $\mathcal{G}$ is linear \eqref{linear}, \eqref{meanfieldstep2} can be written as:
\begin{equation}\label{FKPK}
\left\{
\begin{aligned}
&\partial_t\rho=\nabla\cdot(\rho\Cov_{\rho(t)}\nabla\Phi_R(u))+\text{Tr}\left(\Cov_{\rho(t)}D^2\rho\right)\\
&\rho(u,0)=\rho_0
\end{aligned}\quad,\right.
\end{equation}
where $\Phi_R(u;y)$ is defined in \eqref{PhiRdef}. This means $M_{u_t}$ converges to $\rho$ as $J\to\infty$. The precise statement of this result is in Theorem~\ref{thm:mean_field}.
\item[No. 3:]  The solution to the Fokker-Planck equation converges to the target distribution, meaning $\rho(t)\to\rho_\pos$ as $t\to\infty$. The precise statement of this result is in Theorem~\ref{thm:FKPK_conv}.
\end{itemize}

Finally we combine Theorem~\ref{thm:mean_field} and~\ref{thm:FKPK_conv} and obtain the theorem that justifies the $J,T$ convergence of the EKS (assuming $h=0$). The statement is found in Theorem~\ref{thm:main}.

We emphasize that the goal No. 3 is a direct result of \cite{EnFL,carrillo2019wasserstein}. The main contribution of the paper is to provide the proof for wellposedness (Theorem~\ref{ThforM}) and mean-field limit (Theorem~\ref{thm:mean_field}).

Before presenting the results, we first unify the notations. Throughout the paper we denote
\[
\Cov_{m,n}=\frac{1}{J}\sum^J_{j=1}\left(m^{j}_t-\overline{m}_t\right)\otimes\left(n^{j}_t-\overline{n}_t\right)\,,
\]
the covariance of any vectors $\{m^j\}^J_{j=1}$ and $\{n^j\}^J_{j=1}$, and abbreviate $\ \mathrm{Cov}_m=\Cov_{m,m}$. Here $\otimes$ means the first argument is viewed as a column vector while the second is viewed as the row vector. Similarly, for any probability density function $\rho$ and function $g$, we denote
\[
\EE_{\rho}=\int_{\mathbb{R}^L}u\rho(u)du\,,\quad \EE_{g,\rho}=\int_{\mathbb{R}^L}g(u)\rho(u)du\,,
\]
\[
\Cov_{\rho}=\int_{\mathbb{R}^L}\left(u-\mathbb{E}_\rho\right)\otimes \left(u-\mathbb{E}_\rho\right)\rho(u)du\,,\quad \Cov_{\rho,g}=\int_{\mathbb{R}^L}\left(u-\mathbb{E}_\rho\right)\otimes \left(g(u)-\mathbb{E}_{g,\rho}\right)\rho(u)du\,.
\]
and $\Cov_{g,\rho}=\Cov^\top_{\rho,g}$.

Set $\Omega$ to be the sample space and $\mathcal{F}_0=\sigma\left(u^j(t=0),1\leq j\leq J\right)$, then the filtration introduced by \eqref{CAmunew} is:
\[
\mathcal{F}_t=\sigma\left(u^j(t=0),W^j_s,1\leq j\leq J,s\leq t\right)\,.
\]

The quantity we use to measure the distance between two probability measures is the Wasserstein distance:
\begin{definition} Let $\upsilon_1,\upsilon_2$ be two probability measures in $\left(\mathbb{R}^L,\mathcal{B}_{\mathbb{R}^L}\right)$, then the $W_2$-Wasserstein distance between $\upsilon_1,\upsilon_2$ is defined as
\begin{equation*}
W_2(\upsilon_1,\upsilon_2):=\left(\inf_{\gamma\in\Gamma(\upsilon_1,\upsilon_2)}\int_{\mathbb{R}^L\times\mathbb{R}^L}|x-y|^2d\gamma(x,y)\right)^{\frac{1}{2}},
\end{equation*}
where $\Gamma(\upsilon_1,\upsilon_2)$ denotes the collection of all measures on $\mathbb{R}^L\times\mathbb{R}^L$ with marginals $\upsilon_1$ and $\upsilon_2$ for $x$ and $y$ respectively. Here $\upsilon_i$ can be either general probability measures or the measures induced by probability density functions $\upsilon_i$.
\end{definition}

Now we present the main result.
\begin{theorem}[Main result]\label{thm:main}
Suppose $\mathcal{G}$ satisfies \eqref{linear}, let $\rho(t,u)$ solve~\eqref{FKPK} with initial data $\rho_0$ and $\{u^j_t\}$ solve~\eqref{CAmunew} with $u^j_0$ i.i.d. drawn from the distribution induced by $\rho_0$. Define $M_{u_T}(u)$ to be the ensemble distribution of $\{u^j_T\}$ as in~\eqref{eqn:empirical}. Assume\begin{equation}\label{thmcondition}
\lambda_{\min}(B)\lambda_{\min}\left(\Cov_{\rho(t)}\right)> 1,\quad \forall 0\leq t\leq T\,,
\end{equation}
then for any $0<\delta\ll 1$, there exists $T_\delta>0$ and $J_{T_\delta}>0$ so that
\[
\mathbb{E}(W_2(M_{u_{T_\delta}},\rho_\pos))\leq \delta\,.
\]
\end{theorem}
The introduction of the requirement~\eqref{thmcondition} is a technical one. Essentially it comes from the application of the Ando-Hemmen inequality that studies the differences between two matrices after square roots are taken, see details in the proof of Lemma~\ref{lem:bound_p}. This condition does not seem to be necessary, as long as some reasonable estimates can be found to control the square roots of matrices sensitivity that avoids the application of the Ando-Hemmen inequality. We leave the improvement to future research.

As stated above, this theorem is built upon the following three theorems: 1. Justifying the wellposedness of the SDE system~\eqref{CAmunew} using stochastic Lyapunov theory:
\begin{theorem}\label{ThforM}
Suppose $\mathcal{G}$ satisfies \eqref{linear}, if $\left\{u^j_0:\Omega\rightarrow\mathcal{X}\right\}^J_{j=1}$ is independent almost surely, then for all $t\geq0$, there exists a unique strong solution $(u^j_t)^J_{j=1}$ (up to $\mathbb{P}$-indistinguishability) of the set of coupled SDEs \eqref{CAmunew}.
\end{theorem}

2. Showing the Fokker-Planck equation~\eqref{FKPK} is the mean-field limit of the SDE system~\eqref{Ecov_dis_revise}.
\begin{theorem}\label{thm:mean_field}
Under the same condition as in Theorem~\ref{thm:main}, for any $T>0$ and $0<\epsilon<1/2$, there exits $C$, depending on $T$ and $\epsilon$, but independent of $J$ such that
\[
\begin{aligned}
\mathbb{E}\left(W_2(M_{u},\rho(T,u))\right)\leq C
\left\{
\begin{aligned}
&J^{-1/2+\epsilon},\quad L\leq4\\
&J^{-2/L},\quad L>4\\
\end{aligned}
\right.\,.
\end{aligned}
\]
\end{theorem}

3. Showing the long time convergence to the equilibrium (comes from straightforward derivation from Proposition 3.3 of~\cite{carrillo2019wasserstein}).
\begin{theorem}\label{thm:FKPK_conv}
Let $\rho(t,u)$ solve~\eqref{FKPK} with initial density function $\rho_0\in \mathcal{C}^2$. Suppose
$W_2(\rho(0),\rho_{\text{pos}})<\infty$, then $W_2(\rho(t),\rho_{\text{pos}})$ converge to zero exponentially fast.
\end{theorem}

The proof for the main result is a direct corollary of Theorem~\ref{thm:mean_field} and Theorem~\ref{thm:FKPK_conv}:
\begin{proof}[Proof of Theorem \ref{thm:main}]
For all $0<\delta\ll 1$, according to Theorem~\ref{thm:FKPK_conv}, there exists a time $T_\delta>0$ so that:
\[
W_2(\rho(T_\delta,u),\rho_\pos)\leq \delta/2\,.
\]
For this fixed $T_\delta$, pick any $\epsilon<1/2$, we apply Theorem~\ref{thm:mean_field}, there is a $J_{T_\delta,\epsilon}>0$, such that for any $J>J_{T_\delta,\epsilon}$
\[
\mathbb{E}\left(W_2\left(\rho(T_\delta,u),M_{u_{T_\delta}}\right)\right)\leq \delta/2\,.
\]
The statement of the theorem is immediate with the triangle inequality. In the statement we drop the $\epsilon$ dependence in $J_{T_\delta,\epsilon}$.
\end{proof}

We comment that Theorem~\ref{thm:mean_field} provides the convergence rate. It shows that we have the optimal convergence rate $J^{-1/2}$ in relatively low dimension when $L\leq 4$. In higher dimensional cases, the convergence rate depends on the dimensionality of $u$. We will see in Section~\ref{sec:mean_field} that this is the best possible rate one can get using the approach of the coupling method.

The later two sections are designated to show Theorem~\ref{ThforM} and Theorem~\ref{thm:mean_field}. In particular, we show the wellposedness of the SDE system in Section~\ref{sec:SDE}. We furthermore utilize the results to give some estimates to control the moments of the particle system. In Section~\ref{sec:mean_field}, we show the mean-field limit, Theorem~\ref{thm:mean_field}. We follow the classical trajectorial propagation of chaos approach by inventing a new SDE system, termed $\{v^j_t\}$, as a bridge to connect $\{u^j_t\}$ system and the PDE $\rho$. This section is subsequently divided into two subsections, in which we show the closeness of $\{v^j_t\}$ and $\rho$, and the closeness of the two SDE systems respectively. These two results are Proposition~\ref{prop:vj_FP} and Proposition~\ref{prop:vj_uj}.

\section{Wellposedness of Noisy Ensemble Kalman Flow}\label{sec:SDE}
In this section, we study the wellposedness of the SDE system~\eqref{CAmunew}. Considering each $u^j$ is a vector of $L$-length, we stack them up to have a coupled SDE:
\[
dU_t=F(U_t)dt+G(U_t)dW_t\,,
\]
where $U_t=\left(u^j_t\right)^J_{j=1}\in\RR^{LJ\times1}$, $W_t=\left(W^j_t\right)^J_{j=1}\in\RR^{LJ\times1}$ and
\begin{align*}
&F(U_t)=\left(-\mathrm{Cov}_{u_t}B\left(u^j_t-u^\ast\right)\right)^J_{j=1}\in\RR^{LJ\times1}\,,\\
&G(U_t)=\text{diag}\left(\sqrt{2\mathrm{Cov}_{u_t}}\right)^J_{j=1}\,,\\
\end{align*}
where $\mathrm{Cov}_{u}$ is the empirical covariance and $\text{diag}(D_j)^J_{j=1}$ is a diagonal block matrix with matrices $\left(D_j\right)^J_{j=1}$ on the diagonal and $B$ is defined in~\eqref{def:Bustar}.

We first show the wellposedness of the SDE system~\eqref{CAmunew} by following the standard Lyapunov theory.

\begin{proof}[Proof of Theorem \ref{ThforM}] According to the stochastic Lyapunov theory (See for example Theorem 4.1 \cite{StoBOOk}), strong solution exists if one finds local Lipschitz property of the drift $F$ and the diffusion $G$, namely we need to find a function $V\in C^2\left(\mathbb{R}^{LJ};\mathbb{R}_+\right)$ so that:
\begin{itemize}
\item there is a $c>0$ so that for all $U$:
\begin{equation}\label{WPcondition1}
LV(U):=\nabla V(U)\cdot F(U)+\frac{1}{2}\text{Tr}\left(G^\top(U)\text{Hess}[V](U)G(U)\right)\leq cV(U)\,,
\end{equation}
\item the function blows up at infinity:
\begin{equation}\label{WPcondition2}
\inf_{|U|>R}V(U)\rightarrow\infty\ \text{as}\ R\rightarrow\infty\,.
\end{equation}
\end{itemize}

We define the following Lyapunov function and will justify this function satisfy~\eqref{WPcondition1} and~\eqref{WPcondition2}:
\begin{equation}\label{eqn:3.2V}
V(U)=V_1(U)+V_2(U)=\frac{1}{J}\sum^J_{j=1}|u^j-\bar{u}|^2+|\bar{u}-u^*|^2_B =V_{1}+V_2\,.
\end{equation}
To justify~\eqref{WPcondition1}, we first notice that
\begin{equation}\label{eqn:3.2V1}
\begin{aligned}
&\nabla V_1(U)\cdot F(U)=-\frac{2}{J}\sum^J_{j=1}\left\langle u^j-\bar{u}, \mathrm{Cov}_{u}B(u^j-u^*)\right\rangle=-\frac{2}{J}\sum^J_{j=1}\left\langle u^j-\bar{u}, \mathrm{Cov}_{u}B(u^j-\bar{u})\right\rangle\leq 0\,,\\
&\nabla V_2(U)\cdot F(U)=-\frac{2}{J}\sum^J_{j=1}\left\langle B\left(\bar{u}-u^*\right), \mathrm{Cov}_{u}B(u^j-u^*)\right\rangle=-2\left\langle B\left(\bar{u}-u^*\right), \mathrm{Cov}_{u}B(\bar{u}-u^*)\right\rangle\leq 0\,,
\end{aligned}
\end{equation}
where we used the facts that $\Cov_u$ and $B$ are positive definite, and
\begin{equation}\label{eqn:3.2V2}
\begin{aligned}
&\frac{1}{2}\text{Tr}\left(G^\top(U)\text{Hess}[V_1](U)G(U)\right)=\sum^J_{j=1}\frac{2}{J}\left(1-\frac{1}{J}\right)\left(u^j-\bar{u}\right)^\top\left(u^j-\bar{u}\right)\leq 2V_1(U)\,,\\
&\frac{1}{2}\text{Tr}\left(G^\top(U)\text{Hess}[V_2](U)G(U)\right)=\sum^J_{j=1}\frac{2}{J^2}\left(u^j-\bar{u}\right)^\top B\left(u^j-\bar{u}\right)\leq 2\|B\|_2V_1(U)\,.
\end{aligned}
\end{equation}
Therefore we have
\[
LV(U)=\nabla V(U)\cdot F(U)+\frac{1}{2}\text{Tr}\left(G^\top(U)\text{Hess}[V](U)G(U)\right)\leq 2(1+\|B\|_2)V_1(U)\leq 2(1+\|B\|_2)V(U)\,.
\]
showing \eqref{WPcondition1}. To show \eqref{WPcondition2}, we run the argument of contradiction. Assume there exists $M>0$ and a sequence $\{U_n\}^\infty_{n=1}$ such that
\begin{equation}\label{WPContradic}
\lim_{n\rightarrow\infty} |U_n|=\infty,\quad V_1(U_n)+V_2(U_n)<M\,,
\end{equation}
then
\[
|u^j_n-\bar{u}_n|<\sqrt{MJ},\quad |u^*-\bar{u}_n|<\sqrt{M}\,,
\]
meaning:
\[
|U_n|=\left(\sum^J_{j=1}|u^j_n|^2\right)^{1/2}< \left(\sum^J_{j=1}\left(|u^*|+\sqrt{M}(\sqrt{J}+1)\right)^2\right)^{1/2}\,,
\]
contradicting~\eqref{WPContradic}.
\end{proof}

Preparing to prove the mean-field limit, we now move to show the boundedness of high moments of the SDE system under linear setting \eqref{linear}. In particular, we would like to show that at any finite time $T$, the high moments of $\{u^j_t\}^{J}_{j=1}$ is bounded:
\begin{proposition}\label{BforM}
Suppose $\mathcal{G}$ is linear \eqref{linear}, for the solution $(u^j_t)^J_{j=1}$ of \eqref{CAmunew}, if initial condition has finite higher moments, meaning there exists $M>0$ independent of $J$ such that
\[
\left(\mathbb{E}\left|u^j_0\right|^p\right)^{1/p}<M,\quad\forall 1\leq j\leq J\,
\]
for $p\geq2$, then the boundedness still holds true for any $t\geq0$ and $1\leq j\leq J$, namely:
\begin{itemize}
\item[1.]
\begin{equation}\label{eqn:moment_e}
\left(\mathbb{E}\left|u^j_t-\bar{u}_t\right|^p\right)^{1/p}\leq Ce^{Ct}\,,\quad\text{and}\quad \left(\mathbb{E}\left\|\mathrm{Cov}_{u_t}\right\|^p_2\right)^{1/p}\leq Ce^{Ct}\,,
\end{equation}
\item[2.]
\begin{equation}\label{highmomentu}
\left(\mathbb{E}|u^j_t|^p\right)^{1/p}\leq Ce^{Ce^{Ct}}\,,
\end{equation}
\end{itemize}
with $C>0$ is independent of $J,t$.
\end{proposition}

To show this proposition, we firstly define
\begin{equation}\label{EnsembC}
e^j_t=u^j_t-\bar{u}_t\,.
\end{equation}
Naturally
\[
\left(\mathbb{E}\left|u^j_t-\bar{u}_t\right|^p\right)^{1/p} = \left(\mathbb{E}\left|e^j_t\right|^p\right)^{1/p}\,,\quad\Cov_{u_t} = \Cov_{e_t}\,.
\]
The proof for the boundedness of high moment of $\{u^j_t\}$ now comes down to that for $e^j_t$, as will be shown in the following lemma.

\begin{lem}\label{lem:bound_p1}
Suppose $\mathcal{G}$ is linear \eqref{linear}, if initial $p$-th moment is finite, meaning there is a constant $M>0$ independent of $J$ so that
\begin{equation}\label{leminitial}
\left(\mathbb{E}\left|u^j_0\right|^p\right)^{1/p}<M,\quad\forall 1\leq j\leq J\,,
\end{equation}
for some $p\geq2$, then the boundedness also holds true for $\mathbb{E}|e^j_t|^p$, namely there is $C>0$ depending on $p$ only so that: for any $t\geq 0$ and $1\leq j\leq J$
\begin{equation*}
\left(\mathbb{E}|e^j_t|^p\right)^{1/p}<2(\kappa(B))^{1/2}M\exp(C t)\,,
\end{equation*}
where $B$ is defined in \eqref{def:Bustar} and $\kappa(B)=\|B\|_2/\lambda_{\min}(B)$ is the condition number of $B$ and $\lambda_{\min}$ means the smallest eigenvalue.
\end{lem}

\begin{proof}[Proof of Lemma \ref{lem:bound_p1}]
For convenience, we prove this Lemma for $2p$ with $p\geq 1$. We first define 
\[\te^j_t=\sqrt{B}e^j_t,\quad V_p(\te)=\frac{1}{J}\sum^J_{j=1}\left\langle \te^j_t,\te^j_t\right\rangle^p,
\]
and
\[
h_p(t)=\EE\left(\frac{1}{J}\sum^J_{j=1}\left\langle \te^j_t,\te^j_t\right\rangle^p\right) = \EE V_p\,.
\] 
Because $\lambda_{\min}(B)> 0$, it suffices to prove $h_p(t)$ is bounded.

First, at $t=0$, we have
\[
\begin{aligned}
h^{1/(2p)}_p(0)&=\EE\left(\frac{1}{J}\sum^J_{j=1}\left\langle \te^j_0,\te^j_0\right\rangle^p\right)^{\frac{1}{2p}}\leq \frac{1}{J}\sum^J_{j=1}\EE\left(|\te^j_0|^{2p}\right)^{\frac{1}{2p}}\leq \|B\|_2\EE\left(|e^1_0|^{2p}\right)^{\frac{1}{2p}}\\
&\leq \|B\|^{1/2}_2\left(\EE\left(|u^1_0|^{2p}\right)^{\frac{1}{2p}}+\EE\left(|\overline{u}_0|^{2p}\right)^{\frac{1}{2p}}\right)\leq 2\|B\|^{1/2}_2M\,,
\end{aligned}
\]
where we use triangle inequality in the first inequality, symmetry in the second inequality and \eqref{leminitial} in the last inequality.

According to It\^o's formula, it holds that
\[
dV_p(\te_t)=\sum^J_{j=1}\frac{\partial V_p(\te_t)}{\partial \te^j}d\te^j_t+\frac{1}{2}\sum^J_{i,j=1}\left(d\te^{i}_t\right)^\top\frac{\partial^2 V_p(\te_t)}{\partial \te^j\partial \te^{i}}d\te^j_t\,,
\]
which implies
\[
\begin{aligned}
dV_p(\te_t)=&-\frac{2p}{J}\sum^J_{j=1}\left\langle \te^j_t,\te^j_t\right\rangle^{p-1}\left\langle \te^j_t,\Cov_{\te_t}\te^j_t\right\rangle dt+\frac{2p}{J}\sum^J_{j=1}\left\langle \te^j_t,\te^j_t\right\rangle^{p-1}\left\langle \te^j_t,\sqrt{B}\sqrt{2\Cov_{e_t}}d\left(W^j_t-\overline{W}_t\right)\right\rangle \\
&+\frac{4(J-1)p(p-1)}{J^2}\sum^J_{j=1}\left\langle \te^j_t,\te^j_t\right\rangle^{p-2}\text{Tr}\left\{\left(\te^j_t\otimes \te^j_t\right)\Cov_{\te_t}\right\}dt+\frac{2(J-1)p}{J^2}\sum^J_{j=1}\left\langle \te^j_t,\te^j_t\right\rangle^{p-1}\text{Tr}\left\{\Cov_{\te_t}\right\}dt\,.
\end{aligned}
\]
Then taking the expectation and eliminate the nonpositive first term:
\begin{equation}\label{Vpinequality}
\begin{aligned}
h_p(t)-h_p(0)\leq& \frac{4(J-1)p(p-1)}{J^3}\int^t_0\sum^J_{j,k=1}\EE\left\langle \te^j_s,\te^j_s\right\rangle^{p-2}\left\langle \te^j_s,\te^k_s\right\rangle^2ds+\frac{2(J-1)p}{J^3}\int^t_0\sum^J_{j,k=1}\EE\left\langle \te^j_s,\te^j_s\right\rangle^{p-1}\left\langle \te^k_s,\te^k_s\right\rangle ds\\
\leq& \frac{4(J-1)p(p-1)}{J^3}\int^t_0\sum^J_{j,k=1}\EE\left\langle \te^j_s,\te^j_s\right\rangle^{p-2}\frac{\left\langle \te^j_s,\te^j_s\right\rangle^2+\left\langle \te^k_s,\te^k_s\right\rangle^2}{2}ds\\
&+\frac{2(J-1)p}{J^3}\int^t_0\sum^J_{j,k=1}\EE\left\langle \te^j_s,\te^j_s\right\rangle^{p-1}\left\langle \te^k_s,\te^k_s\right\rangle ds\\
\leq& \frac{2(J-1)p(p-1)}{J}\int^t_0h_p(s)ds+\frac{2(J-1)p(p-1)}{J^3}\int^t_0\sum^J_{j,k=1}\EE\left\langle \te^j_s,\te^j_s\right\rangle^{p-2}\left\langle \te^k_s,\te^k_s\right\rangle^2ds\\
&+\frac{2(J-1)p}{J^3}\int^t_0 \sum^J_{j,k=1}\EE\left\langle \te^j_s,\te^j_s\right\rangle^{p-1}\left\langle \te^k_s,\te^k_s\right\rangle ds\,.
\end{aligned}
\end{equation}
Using the H\"older's inequality we can control the second and third term, namely:
\begin{equation}\label{epexample}
\begin{aligned}
\sum^J_{j,k=1}\EE\left\langle \te^j_s,\te^j_s\right\rangle^{p-2}\left\langle \te^k_s,\te^k_s\right\rangle^2&=\EE\left[\sum^J_{j=1} \left\langle \te^j_s,\te^j_s\right\rangle^{p-2}\right]\left[\sum^J_{k=1} \left\langle \te^k_s,\te^k_s\right\rangle^{2}\right]\\
&\leq J\EE\left[\sum^J_{j=1} \left\langle \te^j_s,\te^j_s\right\rangle^{p}\right]^{(p-2)/p}\left[\sum^J_{k=1} \left\langle \te^k_s,\te^k_s\right\rangle^{p}\right]^{2/p}=J\EE\left[\sum^J_{j=1} \left\langle \te^j_s,\te^j_s\right\rangle^{p}\right]=J^2h_p(t)\,,
\end{aligned}
\end{equation}
and
\begin{equation}\label{epexample2}
\sum^J_{j,k=1}\EE\left\langle \te^j_s,\te^j_s\right\rangle^{p-1}\left\langle \te^k_s,\te^k_s\right\rangle=\EE\left[\sum^J_{j=1} \left\langle \te^j_s,\te^j_s\right\rangle^{p-1}\right]\left[\sum^J_{k=1} \left\langle \te^k_s,\te^k_s\right\rangle\right]\leq J\EE\left[\sum^J_{j=1} \left\langle \te^j_s,\te^j_s\right\rangle^{p}\right]=J^2h_p(t)\,.
\end{equation}
Plug \eqref{epexample}-\eqref{epexample2} into \eqref{Vpinequality}, we finally have $h_p(t)-h_p(0)\leq \frac{(J-1)p(2p-1)}{J}\int^t_0 h_p(s)ds$, which leads to the conclusion using the Gr\"onwall inequality:
\[
h_p(t)\leq h_p(0)e^{\frac{(J-1)p(2p-1)}{J}t}\leq (2M\|B\|^{1/2}_2)^{2p}e^{\frac{(J-1)p(2p-1)}{J}t}\,,
\]
to conclude.
\end{proof}

Now we show Proposition~\ref{BforM}.

\begin{proof}[Proof of Proposition \ref{BforM}]
The first inequality of equation~\eqref{eqn:moment_e} is already shown in Lemma~\ref{lem:bound_p1}. The second inequality is a direct consequence:
\[
\left(\mathbb{E}\left\|\Cov_{u_t}\right\|^p_2\right)^{1/p}\leq \frac{1}{J}\sum^J_{j=1}\mathbb{E}\left(\left\|(u^j_t-\overline{u}_t)\otimes(u^j_t-\overline{u}_t)\right\|^p_2\right)^{1/p}\leq \frac{1}{J}\sum^J_{j=1}\left(\mathbb{E}\left|u^j_t-\overline{u}_t\right|^{2p}\right)^{1/p}\leq Ce^{Ct}\,.
\]
To show~\eqref{highmomentu}, define:
\[
\tu^j_t=\sqrt{B}u^j_t,\quad \tu^*=\sqrt{B}u^*,\quad K_p(\tu)=\frac{1}{J}\sum^J_{j=1}\left\langle \tu^j_t,\tu^j_t\right\rangle^p\,,
\]
and
\[
g_p(t) = \EE\left(\frac{1}{J}\sum^J_{j=1}\left\langle \tu^j_t,\tu^j_t\right\rangle^p\right) = \EE\left(K_p(\tu_t)\right)\,.
\]
Then it's suffices to control the growth of $g(t)$ because $\lambda_{\min}(B)>0$. We first multiply $\sqrt{B}$ onto both sides of \eqref{Ecov_dis_revise} to obtain
\begin{equation}\label{CAmunew2}
d\tu^{j}_t=-\mathrm{Cov}_{\tu_t}(\tu^j_t-\tu^\ast)+\sqrt{B}\sqrt{2\mathrm{Cov}_{u_t}}dW^{j}_t\,.
\end{equation}
Using It\^{o}'s lemma to have:
\[
dK_p(\tu)=\sum^J_{j=1}\frac{\partial K_p(\tu)}{\partial \tu^j}d\tu^j+\frac{1}{2}\sum^J_{i,j=1}d\tu^{i}\frac{\partial^2 K_p(\tu)}{\partial \tu^{i}\partial \tu^j}d\tu^j\,,
\]
which implies
\[
\begin{aligned}
dK_p(\tu_t)=&-\frac{2p}{J}\sum^J_{j=1}\left\langle \tu^j_t,\tu^j_t\right\rangle^{p-1}\left\langle \tu^j_t,\mathrm{Cov}_{\tu_t}\left(\tu^j_t-\tu^*\right)\right\rangle dt+\frac{2p}{J}\sum^J_{j=1}\left\langle \tu^j_t,\tu^j_t\right\rangle^{p-1}\left\langle \tu^j_t,\sqrt{B}\sqrt{2\mathrm{Cov}_{u_t}}dW^j_t\right\rangle \\
&+\frac{4p(p-1)}{J}\sum^J_{j=1}\left\langle \tu^j_t,\tu^j_t\right\rangle^{p-2}\text{Tr}\left\{\left(\tu^j_t\otimes \tu^j_t\right)\mathrm{Cov}_{\tu_t}\right\}dt+\frac{2p}{J}\sum^J_{j=1}\left\langle \tu^j_t,\tu^j_t\right\rangle^{p-1}\text{Tr}\left\{\mathrm{Cov}_{\tu_t}\right\}dt\,.
\end{aligned}
\]
The expectation of the second term vanishes, and to control the first term, we note:
\[
\begin{aligned}
&-\frac{2p}{J}\sum^J_{j=1}\EE\left\langle \tu^j_t,\tu^j_t\right\rangle^{p-1}\left\langle \tu^j_t,\mathrm{Cov}_{\tu_t}\left(\tu^j_t-\tu^*\right)\right\rangle\\
=&-\frac{2p}{J}\sum^J_{j=1}\EE\left\langle \tu^j_t,\tu^j_t\right\rangle^{p-1}\left\langle \tu^j_t,\mathrm{Cov}_{\tu_t}\tu^j_t\right\rangle+\frac{2p}{J}\sum^J_{j=1}\EE\left\langle \tu^j_t,\tu^j_t\right\rangle^{p-1}\left\langle \tu^j_t,\mathrm{Cov}_{\tu_t}\tu^\ast\right\rangle\\
\leq&\,\frac{2p}{J^2}\sum^J_{j,k=1}\EE\left\{\left\langle \tu^j_t,\tu^j_t\right\rangle^{p-1}\left\langle \tu^j_t,\te^k_t\right\rangle\left\langle \te^k_t,\tu^*\right\rangle\right\}\\
\leq&\,\frac{2p}{J^2}\sum^J_{j,k=1}\EE\left\{\left\langle \tu^j_t,\tu^j_t\right\rangle^{p-1}\left|\tu^j_t\right|\left|\te^k_t\right|\left|\te^k_t\right|\left|\tu^*\right|\right\}
\\
=&\,2p\left|\tu^*\right|\EE\left\{\left(\frac{1}{J}\sum^J_{j=1}\left\langle \tu^j_t,\tu^j_t\right\rangle^{p-1/2}\right)\left(\frac{1}{J}\sum^J_{k=1}\left\langle \te^k_t,\te^k_t\right\rangle\right)\right\}\\
\leq&\,2p\left|\tu^*\right|\left(\frac{1}{J}\sum^J_{j=1}\EE\left\langle \tu^j_t,\tu^j_t\right\rangle^{p}\right)^{(p-1/2)/p}\left(\frac{1}{J}\sum^J_{k=1}\EE\left\langle \te^k_t,\te^k_t\right\rangle^{2p}\right)^{1/(2p)}\\
\leq&\,Ce^{Ct}\left(\frac{1}{J}\sum^J_{j=1}\EE\left\langle \tu^j_t,\tu^j_t\right\rangle^{p}\right)^{(p-1/2)/p}=Ce^{Ct}g^{(p-1/2)/p}_p(t)\,,
\end{aligned} 
\]
where the second last inequality comes from H\"older's inequality, and we used the estimate from Lemma~\ref{lem:bound_p1}. To control the third and fourth term, we have:
\[
\begin{aligned}
&\frac{4p(p-1)}{J}\sum^J_{j=1}\EE\left\langle \tu^j_t,\tu^j_t\right\rangle^{p-2}\text{Tr}\left\{\left(\tu^j_t\otimes \tu^j_t\right)\mathrm{Cov}_{\tu_t}\right\}\\
=&\frac{4p(p-1)}{J}\sum^J_{j=1}\EE\left\langle \tu^j_t,\tu^j_t\right\rangle^{p-2}\left(\frac{1}{J}\sum^J_{k=1}\left\langle \tu^j_t,\te^k_t\right\rangle^2\right)\\
\leq&  \frac{4p(p-1)}{J^2}\sum^J_{j,k=1}\EE\left\langle \tu^j_t,\tu^j_t\right\rangle^{p-1}\left\langle \te^k_t,\te^k_t\right\rangle\\
\leq&  4p(p-1)\left(\frac{1}{J}\sum^J_{k=1}\EE\left\langle \tu^j_t,\tu^j_t\right\rangle^{p}\right)^{(p-1)/p}\left(\frac{1}{J}\sum^J_{k=1}\EE\left\langle \te^k_t,\te^k_t\right\rangle^{p}\right)^{1/p}\\
\leq& Ce^{Ct}\left(\frac{1}{J}\sum^J_{k=1}\EE\left\langle \tu^j_t,\tu^j_t\right\rangle^{p}\right)^{(p-1)/p}=Ce^{Ct}g^{(p-1)/p}_p(t)\,,
\end{aligned} 
\]
and
\[
\begin{aligned}
\frac{2p}{J}\sum^J_{j=1}\EE\left\langle \tu^j_t,\tu^j_t\right\rangle^{p-1}\text{Tr}\left\{\mathrm{Cov}_{\tu_t}\right\}&\leq
\frac{2p}{J^2}\sum^J_{j,k=1}\EE\left\langle \tu^j_t,\tu^j_t\right\rangle^{p-1}\left\langle \te^k_t,\te^k_t\right\rangle
\\
&\leq 2p\left(\frac{1}{J}\sum^J_{k=1}\EE\left\langle \tu^j_t,\tu^j_t\right\rangle^{p}\right)^{(p-1)/p}\left(\frac{1}{J}\sum^J_{k=1}\EE\left\langle \te^k_t,\te^k_t\right\rangle^{p}\right)^{1/p}\\
&\leq Ce^{Ct}\left(\frac{1}{J}\sum^J_{k=1}\EE\left\langle \tu^j_t,\tu^j_t\right\rangle^{p}\right)^{(p-1)/p}=Ce^{Ct}g^{(p-1)/p}_p(t)\,.
\end{aligned} 
\]
In conclusion, we obtain
\[
\frac{dg}{dt}\leq Ce^{Ct}\left[g^{(p-1)/p}_p(t)+g^{(p-1/2)/p}_p(t)\right]\quad\Rightarrow\quad g_p(t)\leq g_p(t=0)Ce^{Ce^{Ct}}\,.
\]
\end{proof}

\section{Mean-field limit of ~\eqref{CAmunew}}\label{sec:mean_field}
In this section we show that the mean-field limit of~\eqref{CAmunew} is the Fokker-Planck equation~\eqref{FKPK}, and prove Theorem~\ref{thm:mean_field}. As discussed in the Introduction, the approach we utilize is the classical method termed the trajectorial propagation of chaos. With this approach: one builds a completely new SDE system according to the limiting PDE~\eqref{FKPK} by utilizing exactly the same coefficients, and compare the newly build SDE with the given SDE~\eqref{CAmunew}. Since the newly built SDE follows exactly the same flow as the PDE, its ensemble distribution is expected to be close to the PDE solution. And by inventing a new SDE system, it makes it easier analytically for the comparison.

More specifically, for the case studied in this paper, derived from~\eqref{FKPK}, we develop the SDE system $\{v^j_t\}$ that satisfy:
\begin{equation}\label{eqn:vj}
dv^j_t=-\Cov_{\rho(t)}\nabla \Phi_R(v^j_t)dt+\sqrt{2\Cov_{\rho(t)}}dW^j_t,\quad 1\leq j\leq J\,,
\end{equation}
with $v^j_0 = u^j_0$ drawn from the distribution induced by $\rho_0$. The corresponding ensemble distribution is:
\begin{equation}\label{eqn:ensemble_v}
M_{v_t}(u)=\frac{1}{J}\sum^{J}_{j=1}\delta_{v^{j}_t}(u)\,.
\end{equation}

In the following two subsections respectively, we first study the closeness of $\rho$ with $M_{v_t}$, and then compare the two SDE systems,~\eqref{CAmunew} v.s.~\eqref{eqn:vj} and justify the smallness between $M_{v_t}$ and $M_{u_t}$. The two results are stated in the following two Propositions respectively.
\begin{proposition}\label{prop:vj_FP}
Let $\rho$ is the solution to the Fokker-Planck equation~\eqref{FKPK}, and let $\{v^j\}$ solve \eqref{eqn:vj}, with initial data $\{v^j_{t=0}\}$ drawn i.i.d. from the distribution induced by $\rho_0$. Suppose $\rho_0$ is a $\mathcal{C}^2$ function and has finite higher moments, then for any $t>0$ and $0<\epsilon<1/2$, there exists a constant $C$, depending on $t$, dimension $L$ and $\epsilon$ but not on $J$ such that
\begin{equation}\label{eqn:vj_FP}
\mathbb{E}\left(W_2(M_{v_t},\rho(t,u))\right)\leq C
\left\{
\begin{aligned}
&J^{-1/2+\epsilon},\quad L\leq4\\
&J^{-2/L},\quad L>4
\end{aligned}
\right.\,.
\end{equation}
\end{proposition}

\begin{proposition}\label{prop:vj_uj}
Let $\{u^j_t\}$ solve~\eqref{CAmunew} and $\{v^j_t\}$ solve~\eqref{eqn:vj}, with its coefficient defined by $\rho$, the solution to~\eqref{FKPK}. Suppose~\eqref{thmcondition} holds true, and $u^j_0=v^j_0$ are i.i.d. drawn from the distribution induced by $\rho_0$ ($\mathcal{C}^2$ and has finite high moments), then for any $0<\epsilon<1/2$, there exists a constant $C$ depending only on $L$, $T$ and $\epsilon$ such that
\begin{equation}\label{eqn:W2_estimate}
\mathbb{E}\left(W_2(M_{v_T},M_{u_T})\right)\leq\left(\frac{1}{J}\sum^J_{j=1}\mathbb{E}|u^j_T-v^j_T|^2\right)^{1/2}\leq C{J^{-1/2+\epsilon}}\,.
\end{equation}
\end{proposition}

The proof for Theorem~\ref{thm:mean_field} is then natural:
\begin{proof}[Proof of Theorem~\ref{thm:mean_field}]
Considering~\eqref{eqn:vj_FP} and~\eqref{eqn:W2_estimate} and, by triangle inequality, one has: for any $0<\epsilon<1/2$
\[
\begin{aligned}
\mathbb{E}\left(W_2(M_{u},\rho(T,u))\right)&\leq \mathbb{E}\left(W_2(M_{u},M_{v})\right)+\mathbb{E}\left(W_2(M_{v},\rho(T,u))\right)\\
&\leq C
\left\{
\begin{aligned}
&J^{-1/2+\epsilon},\quad L\leq4\\
&J^{-2/L},\quad L>4\\
\end{aligned}
\right.\,.
\end{aligned}
\]
with $C$ independent of $J$. Setting this less than $\epsilon$ gives $J_\epsilon$ which concludes.
\end{proof}

\subsection{Comparing $\rho(t,u)$ and $M_{v_t}$}\label{sec:PDE}
In this section, we study the closeness of the limiting PDE~\eqref{FKPK} with its i.i.d. samples, the $\{v^j\}$ system. The goal is to prove Proposition~\ref{prop:vj_FP}.

To show this proposition, we first cite a classical result that states that the ensemble distribution of $i.i.d.$ samples approximates the original measure is indeed close:
\begin{theorem}[Theorem 1 in~\cite{Fournier2015}]\label{Fournier}
Let $\rho(u)$ be a probability density function on $\mathbb{R}^L$ and let $p>0$. Assume that
\begin{equation}\label{eqn:bound_moment_Fournier}
M_q(\rho):=\int_{\mathbb{R}^d}|x|^q\rho(dx)<\infty
\end{equation}
for some $q>p$. Consider an $i.i.d.$ sequence $(X_k)_{k\geq1}$ sampled from distribution induced by $\rho(u)$ and, for $J\geq1$, define the empirical measure
\[
\rho_J:=\frac{1}{J}\sum^J_{k=1}\delta_{X_k}.
\]
Then for all $J\geq1$ and $0<\epsilon\ll 1$, there exists a constant $C$ depending only on $p,L,q,\epsilon$ such that
\[
\mathbb{E}\left(W_p(\rho_J,\rho)\right)\leq CM^{p/q}_q(\rho)
\left\{
\begin{aligned}
&J^{-1/2+\epsilon}+J^{-(q-p)/q},\quad if\ p\geq L/2\ and\ q\neq 2p\\
&J^{-p/L}+J^{-(q-p)/q},\quad p\in(0,L/2),\quad if\ p\in(0,L/2)\ and\ q\neq L/(L-p)
\end{aligned}
\right..
\]
\end{theorem}

Our Proposition~\ref{prop:vj_FP} can be viewed as a direct corollary of this theorem if one can show the boundedness of the moment~\eqref{eqn:bound_moment_Fournier} for a large enough $q$ (setting $p=2$). This makes the second term in Theorem~\ref{Fournier} vanish and we get a simpler version as shown in~\eqref{eqn:vj_FP}. The rest of the subsection is dedicated to the boundedness of the moments.

We first cite results from~\cite{EnFL} and~\cite{carrillo2019wasserstein}:
\begin{lem}[Proposition 4 from~\cite{EnFL} and (2.2) from~\cite{carrillo2019wasserstein}]\label{lem2:EnFL}
Suppose $\mathcal{G}$ is linear \eqref{linear}, let $\rho(t,u)$ solve \eqref{FKPK} with initial density $\rho_0$ that is a $\mathcal{C}^2$ function and has finite second moments, then the mean $\mathsf{m}$ and the covariance $\mathsf{C}$ of the solution to~\eqref{FKPK} is governed by
\begin{equation}\label{eqn:mc}
\frac{d}{dt}\mathsf{m}(t)=-\mathsf{C}(t)(B\mathsf{m}(t)-r)\,,\quad \frac{d}{dt}\mathsf{C}(t)=-2\mathsf{C}(t)B\mathsf{C}(t)+2\mathsf{C}(t)\,.
\end{equation}
Furthermore, we have \begin{equation}\label{eqn:c}
\mathsf{C}(t)=\left((1-e^{-2\sigma t})B+e^{-2\sigma t}\mathsf{C}^{-1}(0)\right)^{-1}\,,
\end{equation}
where $B$ is defined in \eqref{def:Bustar} and $\mathsf{m}(t)\rightarrow \EE_{\rho_{\mathrm{pos}}}$, $\mathsf{C}(t)\rightarrow \mathrm{Cov}_{\rho_{\mathrm{pos}}}$ exponentially as $t\rightarrow\infty$ .
\end{lem}

Then, since the covariance of solution to the PDE is known, we can easily obtain upper bounds for higher moments:
\begin{lem}\label{lem3:EnFL}
If $\rho_0\in\mathcal{C}^2$ and has finite high moments, then for any $p\geq2$, $t>0$, there exists a constant $C$ depending on $p$ and $t$ such that
\begin{equation}\label{eqn:covbound}
\int |u|^p\rho(t,u)du\leq C(p,t)<\infty,\quad\text{and}\quad\|\Cov_{\rho(t)}\|^p_2\leq C(p,t)<\infty\,,
\end{equation}
\end{lem}
\begin{proof}
According to Lemma \ref{lem2:EnFL} \eqref{eqn:c}, the covariance of $\rho(t,u)$ is uniformly bound, namely:
\[
\|\Cov_{\rho(t)}\|_F\leq M,\quad \forall t>0\,,
\]
for a constant $M$ independent of $t$. This means the transport coefficient of \eqref{FKPK} is Lipschitz and Hessian coefficient of \eqref{FKPK} is uniformly bounded, considering the formula in~\eqref{PhiRdef}:
\[
\|\Cov_{\rho(t)}\nabla_u\Phi_R(u)\|_2=\|2\Cov_{\rho(t)}\left[B(u-u^*)\right]\|_2=2\|\Cov_{\rho(t)}B\|_2|u-u^\ast|\leq 2\|B\|_2M|u-u^\ast|\,,
\]
for all $t>0$ and this implies, using~\eqref{FKPK}, that high moments of $\rho(t)$ are also finite for any time $t<\infty$.
\end{proof}

Naturally one can prove Proposition~\ref{prop:vj_FP}
\begin{proof}[Proof of Proposition \ref{prop:vj_FP}]
Since~\eqref{eqn:bound_moment_Fournier} holds true according to \eqref{highmomentv}, we conclude the proof by choosing $p=2$ and $q$ large enough in Theorem \ref{Fournier}.
\end{proof}

For later convenience we also provide the boundedness of the moments for $M_{v_t}$.

\begin{proposition}\label{Bforv}
Let $\rho$ solve~\eqref{FKPK} with the initial data $\rho_0\in\mathcal{C}^2$ and has finite high moments, and let $v^j_t$ solve the SDE system~\eqref{eqn:vj}, then for any $J$, the bound holds true for all finite time $t$, namely there is $C>0$ depending on $p,M,t$ so that for all $1\leq j\leq J$:
\begin{equation}\label{highmomentv}
\left(\mathbb{E}|v^j_t|^p\right)^{1/p}\leq C\,,\quad \left(\mathbb{E}\left\|\mathrm{Cov}_v(t)\right\|^p_2\right)^{1/p}\leq C\,,\quad \left(\mathbb{E}\left|v^j_t-\bar{v}_t\right|^p\right)^{1/p}\,\leq C\,.
\end{equation}
Furthermore we have 
\begin{equation}\label{diffforconv}
\left(\mathbb{E}\left\|\Cov_{v_t}-\Cov_{\rho(t)}\right\|^p_2\right)^{1/p}\leq CJ^{-1/2}\,.
\end{equation}
and
\begin{equation}\label{highmomentv4}
\begin{aligned}
\left(\mathbb{E}\left\|\overline{v}-\EE_{\rho(t)}\right\|^p_2\right)^{1/p}\leq CJ^{-1/2}\,,\quad \left(\mathbb{E}\left\|\frac{1}{J}\sum^J_{j=1}|q^j|^2-\mathrm{Var}_{\rho(t)}\right\|^p_2\right)^{1/p}\leq CJ^{-1/2}\,,
\end{aligned}
\end{equation}
where $\mathrm{Var}_{\rho(t)}=\mathrm{Tr}\left(\Cov_{\rho(t)}\right)$.
\end{proposition}
The proof is rather tedious but not very insightful. We leave it to Appendix \ref{ap:proofofBforv}.

\subsection{Comparing $\{v^j_t\}$ and $\{u^j_t\}$ systems}\label{sec:meanfield}
In this section we show that the two particle systems are asymptotically equivalent, namely Proposition~\ref{prop:vj_uj}. More specifically, $\{u^j\}$ system is governed by a coupled SDE~\eqref{CAmunew}, while $\{v^j\}$ comes from i.i.d. sampling of the Fokker-Planck equation~\eqref{FKPK} and is governed by~\eqref{eqn:vj}. We will show the $W_2$-Wasserstein distance of the ensemble distribution of $\{v_j\}$ and $\{u_j\}$ converge in $J$ for all $t>0$. This kind of techniques are widely used in many applications such as~\cite{Bolley_Carrillo,CA_IZO_2011, ding2019meanfield, LuLuNolen,glose,Sznitman,jin2018random,JABIN20163588,1937-5093_2014_4_661,Serfaty-2015} and particle method for PDE \cite{FBP,Hou,CHERTOCK2001708,Blob}.

This proposition is a consequence of a few lemmas. We firstly define the distance of the two particle systems:
\begin{equation}\label{eqn:def_x}
x^j_t=u^j_t-v^j_t\,,\quad p^j_t=x^j_t-\overline{x}_t\,,\quad q^j_t=v^j_t-\overline{v}_t\,,
\end{equation}
then we have
\begin{equation}\label{covequation}
\Cov_{u_t}=\Cov_{x_t+v_t}=\Cov_{p_t+q_t},\quad \Cov_{v_t}=\Cov_{q_t},\quad \Cov_{x_t}=\Cov_{p_t}\,.
\end{equation}

We will show in Lemma~\ref{cor:prebound} that the moment of $x^j_t$ is bounded for all time. Then in Lemma~\ref{lem:bound_x} we will show that if the second moment of $x^j_t$ decays with a certain rate $J^{-\alpha}$, where $0\leq\alpha<1/2$, the decay rate can be tightened to $J^{-1/2-\alpha/2+\epsilon}$. According to Lemma~\ref{cor:prebound}, this $\alpha$ is at least $0$, and then we use Lemma~\ref{lem:bound_x} to iterate till we obtain the optimal convergence rate $J^{-1/2+\epsilon}$.

As discussed in the Introduction, this bootstrapping argument is not seen often in the mean-field proofs, mostly because in previous systems some kind of Lipschitz condition is imposed on the coefficient, which immediately prompts the Gr\"onwall inequality to loop back the bound.  When nonlinearity presents, such as in~\cite{huang2018meanfield,Lazarovici2015AMF}, one draws a large domain for the Lipschitz condition to hold true inside and separate the discussions. All these are done assuming the coefficients in the Brownian motion are constants, and can be canceled out when two SDE systems are compared. This allows $L^\infty$ type boundedness. When the Brownian motion coefficients are also functionals of $u$, a comparison cannot eliminate the Brownian motion, and $L^\infty$ estimate has to be replaced by other norms. This difficulty has been encountered in Mckean-Vlasov system as studied in~\cite{Sznitman,Mlard1996}. But to the best of our knowledge, Lipschitz continuity is used.

Now we state the first lemma.
\begin{lem}\label{cor:prebound}
Let $\{u^j_t\}$ solve \eqref{CAmunew} and $\{v^j_t\}$ solve \eqref{eqn:vj} with same initial condition. The coefficient for $v^j_t$ are determined by $\rho$, the solution to~\eqref{FKPK} with the initial data $\rho_0\in\mathcal{C}^2$ that has finite high moments. Let $x^j_t$ be defined as in~\eqref{eqn:def_x}, then under condition \eqref{thmcondition}, for all $2\leq p<\infty$ and $T>0$, we have a constant $C_p$ independent of $J,t$ such that: 
\begin{equation}\label{prebound}
\mathbb{E}|x^j_t|^p=\mathrm{E}|x^1_t|^p\leq C_p\,,\ \mathbb{E}|p^j_t|^p=\mathrm{E}|p^1_t|^p\leq C_p\,.
\end{equation}
for all $1\leq j\leq J$ and $0\leq t\leq T$.
\end{lem}
\begin{proof} The first inequality is a direct result from the fact that
\[
\left(\mathbb{E}|x^j_t|^p\right)^{1/p}\leq \left(\mathbb{E}|u^j_t|^p\right)^{1/p}+\left(\mathbb{E}|v^j_t|^p\right)^{1/p}
\]
and then applying Proposition \ref{BforM} \eqref{highmomentu} and Proposition~\ref{Bforv} \eqref{highmomentv}. Then the second inequality comes from 
\[
\left(\mathbb{E}|p^j|^p\right)^{1/p}\leq \left(\EE|x^j_t|^p\right)^{1/p}+\left(\EE|\overline{x}_t|^p\right)^{1/p}\leq \left(\EE|x^j_t|^p\right)^{1/p}+\frac{1}{J}\sum^J_{j=1}\left(\EE|x^j_t|^p\right)^{1/p}\leq 2\left(\mathbb{E}|x^1|^p\right)^{1/p}\,.
\]
\end{proof}

The iterative lemma is now presented. Firstly:
\begin{lem}\label{lem:bound_p}
Let $\{u^j_t\}$ solve \eqref{CAmunew} and $\{v^j_t\}$ solve \eqref{eqn:vj} with same initial condition. The coefficient for $v^j_t$ are determined by $\rho$, the solution to~\eqref{FKPK} with the initial data $\rho_0\in\mathcal{C}^2$ that has finite high moments. Let $x^j_t$ be defined as in~\eqref{eqn:def_x}, then under condition \eqref{thmcondition}, for any $0\leq\alpha<1$ and $T>0$, if there is a constant $C$ independent of $J,t$ so that
\begin{equation}\label{preboundconditionforx}
\mathbb{E}|x^j_t|^2\leq CJ^{-\alpha}\,,
\end{equation}
for all $1\leq j\leq J$ and $0\leq t\leq T$, then we can tighten the decay rate, namely: for any $0<\epsilon<1/2$ and $1\leq j\leq J$, there is a constant $\tilde{C}$ independent of $J,t$ so that
\begin{equation}\label{pestimation}
\mathbb{E}\left|p^j_t\right|^2=\mathbb{E}\left|x^j_t-\frac{1}{J}\sum^J_{j=1}x^k_t\right|^2\leq \tilde{C}J^{-1/2-\alpha/2+\epsilon}\,.
\end{equation}
for all $1\leq j\leq J$ and $0\leq t\leq T$.
\end{lem}

Then we have:
\begin{lem}\label{lem:bound_x}
Under the same condition as in Lemma~\ref{lem:bound_p}, we have for any $0<\epsilon<1/2$ and $T>0$, there is a constant $\tilde{C}$ independent of $J,t$ so that
\begin{equation}\label{xestimation}
\mathbb{E}|x^j_t|^2\leq \tilde{C}J^{-1/2-\alpha/2+\epsilon}\,.
\end{equation}
for all $1\leq j\leq J$ and $0\leq t\leq T$.
\end{lem}

This lemma, when combined with Lemma~\ref{cor:prebound} immediately allows us to show Proposition~\ref{prop:vj_uj}.

\begin{proof}[Proof of Proposition \ref{prop:vj_uj}]
First, by Corollary \ref{cor:prebound}, we have the condition \eqref{preboundconditionforx} holds true for $\alpha_0=0$. Then Lemma \ref{lem:bound_x} implies \eqref{preboundconditionforx} is true for $\alpha_1=1/2-\epsilon$ for any small $\epsilon>0$. Recursively:
\begin{equation*}
\alpha_n=1/2+\alpha_{n-1}/2-\epsilon\,.
\end{equation*}
Since $\lim_{n\rightarrow\infty}\alpha_n=1-2\epsilon$, \eqref{preboundconditionforx} holds true with $\alpha=1-2\epsilon$ for any $\epsilon>0$, and this completes the proof. 
\end{proof}

Now we prove the two lemmas.
\begin{proof}[Proof of Lemma~\ref{lem:bound_p}]
First of all, due to the symmetry of the particle system, for all $1\leq j\leq J$ and $0\leq t\leq 1$:
\begin{equation*}
\mathbb{E}|p^j_t|^2=\mathbb{E}|p^1_t|^2\,,\quad \mathbb{E}|x^j_t|^2=\mathbb{E}|x^1_t|^2\,.
\end{equation*}
Then condition \eqref{preboundconditionforx} implies
\begin{equation}\label{eqn:pre_bound_p}
\left(\mathbb{E}|p^j_t|^2\right)^{1/2}\leq \left(\EE|x^j_t|^2\right)^{1/2}+\left(\EE|\overline{x}_t|^2\right)^{1/2}\leq 2\left(\mathbb{E}|x^1_t|^2\right)^{1/2}\leq 2CJ^{-\alpha/2}\,.
\end{equation}

Subtracting the SDEs~\eqref{CAmunew} and~\eqref{eqn:vj}, we have
\begin{equation}\label{xsde}
\begin{aligned}
dx^j_t=&\left(-\Cov_{x_t+v_t}B(x^j_t+v^j_t)+\Cov_{\rho(t)}Bv^j_t\right)dt+\left(\Cov_{x_t+v_t}-\Cov_{\rho(t)}\right)Bu^*dt\\
&+\left(\sqrt{2\Cov_{x_t+v_t}}-\sqrt{2\Cov_{\rho(t)}}\right)dW^j_t\,.
\end{aligned}
\end{equation}

Using Ito's formula, this becomes
\[
\begin{aligned}
d|x^j_t|^2=&-2\left\langle x^j_t,\Cov_{x_t+v_t}Bx^j_t\right\rangle dt-2\left\langle x^j_t,\left(\Cov_{x_t+v_t}-\Cov_{\rho(t)}\right)Bv^j_t\right\rangle dt\\
&+2\left\langle x^j_t,\left(\Cov_{x_t+v_t}-\Cov_{\rho(t)}\right)Bu^*\right\rangle dt+2\mathrm{Tr}\left(\sqrt{\Cov_{x_t+v_t}}-\sqrt{\Cov_{\rho(t)}}\right)^2dt\\
&+2\left\langle x^j_t,\left(\sqrt{2\Cov_{x_t+v_t}}-\sqrt{2\Cov_{\rho(t)}}\right)dW^j_t\right\rangle\,.
\end{aligned}
\]

Replace $\Cov_{\rho(t)}$ with $\Cov_{v_t}$ in second and third terms, we obtain
\begin{equation}\label{x2diff}
\begin{aligned}
d|x^j_t|^2=&-2\left\langle x^j_t,\Cov_{x_t+v_t}Bx^j_t\right\rangle dt-2\left\langle x^j_t,\left(\Cov_{x_t+v_t}-\Cov_{v_t}\right)Bv^j_t\right\rangle dt\\
&+2\left\langle x^j_t,\left(\Cov_{x_t+v_t}-\Cov_{v_t}\right)Bu^*\right\rangle dt+2\mathrm{Tr}\left(\sqrt{\Cov_{x_t+v_t}}-\sqrt{\Cov_{\rho(t)}}\right)^2dt\\
&+2\left\langle x^j_t,\left(\sqrt{2\Cov_{x_t+v_t}}-\sqrt{2\Cov_{\rho(t)}}\right)dW^j_t\right\rangle+\mathrm{R}^j_tdt\,,
\end{aligned}
\end{equation}
where the remainder $R^j_t$ is introduced to account for the replacement:
\[
R^j_t=2\left\langle x^j_t,\left(\Cov_{\rho(t)}-\Cov_{v_t}\right)Bv^j_t\right\rangle-2\left\langle x^j_t,\left(\Cov_{\rho(t)}-\Cov_{v_t}\right)Bu^*\right\rangle\,.
\]

We then take average of~\eqref{xsde} in $j$ to obtain
\[
\begin{aligned}
d\overline{x}_t=&\left(-\Cov_{x_t+v_t}B(\overline{x}_t+\overline{v}_t)+\Cov_{\rho(t)}B\overline{v}_t\right)dt+\left(\Cov_{x_t+v_t}-\Cov_{\rho(t)}\right)Bu^*dt\\
&+\left(\sqrt{2\Cov_{x_t+v_t}}-\sqrt{2\Cov_{\rho(t)}}\right)d\overline{W}_t\,,
\end{aligned}
\]
so that according to Ito's formula:
\begin{equation}\label{overlinex2diff}
\begin{aligned}
d|\overline{x}_t|^2=&-2\left\langle \overline{x}_t,\Cov_{x_t+v_t}B\overline{x}_t\right\rangle dt-2\left\langle \overline{x}_t,\left(\Cov_{x_t+v_t}-\Cov_{v_t}\right)B\overline{v}_t\right\rangle dt\\
&+2\left\langle \overline{x}_t,\left(\Cov_{x_t+v_t}-\Cov_{v_t}\right)Bu^*\right\rangle dt+\frac{2}{J}\mathrm{Tr}\left(\sqrt{\Cov_{x_t+v_t}}-\sqrt{\Cov_{\rho(t)}}\right)^2dt\\
&+2\left\langle \overline{x}_t,\left(\sqrt{2\Cov_{x_t+v_t}}-\sqrt{2\Cov_{\rho(t)}}\right)d\overline{W}_t\right\rangle+\overline{R}_tdt,
\end{aligned}
\end{equation}
where the remainder term:
\[
\overline{R}_t=2\left\langle \overline{x}_t,\left(\Cov_{\rho(t)}-\Cov_{v_t}\right)B\overline{v}_t\right\rangle-2\left\langle \overline{x}_t,\left(\Cov_{\rho(t)}-\Cov_{v_t}\right)Bu^*\right\rangle\,.
\]

Combine~\eqref{x2diff} and~\eqref{overlinex2diff}, it is a straightforward calculation that:
\begin{equation}\label{xoverlinex2diff}
\begin{aligned}
d\left(\frac{1}{J}\sum^J_{j=1}|x^j_t|^2-|\overline{x}_t|^2\right)=&-\frac{2}{J}\sum^J_{j=1}\left\langle p^j_t,\Cov_{p_t+q_t}Bp^j_t\right\rangle dt-\frac{2}{J}\sum^J_{j=1}\left\langle p^j_t,\left(\Cov_{p_t+q_t}-\Cov_{q_t}\right)Bq^j_t\right\rangle dt\\
&+2\left(1-\frac{1}{J}\right)\mathrm{Tr}\left(\sqrt{\Cov_{x_t+v_t}}-\sqrt{\Cov_{\rho(t)}}\right)^2dt+\left(\frac{1}{J}\sum^J_{j=1}R^j_t-\overline{R}_t\right)dt\\
&+\frac{2}{J}\sum^J_{j=1}\left\langle \left(x^j_t-\overline{x}_t\right),\left(\sqrt{2\Cov_{x_t+v_t}}-\sqrt{2\Cov_{\rho(t)}}\right)d\left(W^j_t-\overline{W}_t\right)\right\rangle\,.
\end{aligned}
\end{equation}

According to the definition of $p^j_t$, taking the expectation of~\eqref{xoverlinex2diff} we have $d\mathbb{E}\left|p^j_t\right|^2$.

The expectation of the last term is $0$ due to the property of the Brownian motion. Since $R^j_t$ involves the difference between $\Cov_\rho$ and $\Cov_{v}$, it is expected that the second last term can be controlled using the central limit theorem. Indeed:
\begin{equation}\label{diffR}
\begin{aligned}
\EE\left(\frac{1}{J}\sum^J_{j=1}R^j_t-\overline{R}_t\right)&=2\EE\frac{1}{J}\sum^J_{j=1}\left\langle p^j_t,\left(\Cov_{\rho(t)}-\Cov_{v_t}\right)Bq^j_t\right\rangle=2\EE\left\langle p^1_t,\left(\Cov_{\rho(t)}-\Cov_{v_t}\right)Bq^1_t\right\rangle\\
&\leq 2\left(\EE\|\Cov_{\rho(t)}-\Cov_{v_t}\|^2_2\right)^{1/2}\left(\EE\|p^1_t\|^2\|Bq^1_t\|^2\right)^{1/2}\\
&\leq 2\left(\EE\|\Cov_{\rho(t)}-\Cov_{v_t}\|^2_2\right)^{1/2}\left(\EE\|p^1_t\|^{2-\epsilon}\|p^1_t\|^{\epsilon}\|Bq^1_t\|^2\right)^{1/2}\\
&\leq2\left(\EE\|\Cov_{\rho(t)}-\Cov_{v_t}\|^2_2\right)^{1/2}\left(\EE\|p^1_t\|^{2}\right)^{(2-\epsilon)/4}\left(\EE\|p^1_t\|^{2}\|Bq^1_t\|^{4/\epsilon}\right)^{\epsilon/4}\\
&\leq 2\left(\EE\|\Cov_{\rho(t)}-\Cov_{v_t}\|^2_2\right)^{1/2}\left(\EE\|p^1_t\|^{2}\right)^{(2-\epsilon)/4}\left(\EE\|p^1_t\|^{4}\right)^{\epsilon/8}\left(\EE\|Bq^1_t\|^{8/\epsilon}\right)^{\epsilon/8}\\
&\leq \frac{C}{J^{1/2}}\left(\EE\|p^1_t\|^{2}\right)^{(2-\epsilon)/4}\leq C J^{-1/2-\alpha/2+\epsilon\alpha/4}\,,
\end{aligned}
\end{equation}
where we use symmetry in the second equality and the $-1/2$ rate comes from~\eqref{diffforconv}, and the $-\alpha/2+\epsilon\alpha/4$ rate comes from~\eqref{eqn:pre_bound_p}. The uniform boundedness of high moments are stated in Lemma~\ref{cor:prebound}, equation~\eqref{prebound}, and Proposition~\ref{Bforv}, equation~\eqref{highmomentv}-\eqref{highmomentv4}. The constant here depends on $\epsilon$.

The third term in~\eqref{xoverlinex2diff} is expected to contribute a relatively slow-decaying term. For that, we apply Ando-Hemmen inequality (see for instance Theorem 6.2 on page 135 in \cite{Matrixroot}). Define $\lambda_0 = \lambda_{\min}\left(\Cov_{\rho(t)}\right)$, then:
\begin{equation}\label{Trterm1}
\begin{aligned}
&\EE\mathrm{Tr}\left(\sqrt{\Cov_{x_t+v_t}}-\sqrt{\Cov_{\rho(t)}}\right)^2=\EE\left\|\sqrt{\Cov_{x_t+v_t}}-\sqrt{\Cov_{\rho(t)}}\right\|^2_F\\
\leq&\EE\left[\frac{1}{\lambda_0}\left\|\Cov_{x_t+v_t}-\Cov_{\rho(t)}\right\|^2_F\right]\\
\leq&\frac{1}{\lambda_0}\left\{\EE\left\|\Cov_{x_t+v_t}-\Cov_{v_t}\right\|^2_F+\EE\left\|\Cov_{v_t}-\Cov_{\rho(t)}\right\|^2_F\right\}\\
&+\frac{2}{\lambda_0}\left(\EE\left\|\Cov_{x_t+v_t}-\Cov_{v_t}\right\|^2_F\right)^{1/2}\left(\EE\left\|\Cov_{v_t}-\Cov_{\rho(t)}\right\|^2_F\right)^{1/2}\\
\leq&\frac{1}{\lambda_0}\EE\left\|\Cov_{x_t+v_t}-\Cov_{v_t}\right\|^2_F+\frac{2J^{-1/2}}{\lambda_0}\left(\EE\left\|\Cov_{x_t+v_t}-\Cov_{v_t}\right\|^2_F\right)^{1/2}+\frac{J^{-1}}{\lambda_0}\,,
\end{aligned}
\end{equation}
where we used Proposition~\ref{Bforv} to control $\left(\EE\left\|\Cov_{v_t}-\Cov_{\rho(t)}\right\|^2_F\right)^{1/2}$. To estimate $\left(\EE\left\|\Cov_{x_t+v_t}-\Cov_{v_t}\right\|^2_F\right)^{1/2}$, we cite Lemma~\ref{secondlemma1} in Appendix~\ref{sec:twolemmas}, and use~\eqref{eqn:pre_bound_p} for \eqref{eqn:lemma2}:
\[
\left(\EE\left\|\Cov_{x_t+v_t}-\Cov_{v_t}\right\|^2_F\right)^{1/2}\leq C J^{-\alpha/2+\epsilon\alpha/4}\,.
\]
Here the constant $C$ only depend on $\epsilon$.

Plug this in~\eqref{Trterm1} to replace the second term, we simplify it to:
\begin{equation}\label{traceterm}
\begin{aligned}
\EE\mathrm{Tr}\left(\sqrt{\Cov_{x_t+v_t}}-\sqrt{\Cov_{\rho(t)}}\right)^2\leq&\frac{1}{\lambda_0}\EE\left\|\Cov_{p_t+q_t}-\Cov_{q_t}\right\|^2_F+C_\epsilon J^{-1/2-\alpha/2+\epsilon\alpha/4}\,.
\end{aligned}
\end{equation}

Finally, we deal with first and second term in \eqref{xoverlinex2diff}, we first rewrite:
\begin{equation*}
\begin{aligned}
&-\frac{2}{J}\EE\left[\sum^J_{j=1}\left\langle p^j_t,\Cov_{p_t+q_t}Bp^j_t\right\rangle+\sum^J_{j=1}\left\langle p^j_t,\left(\Cov_{p_t+q_t}-\Cov_{q_t}\right)Bq^j_t\right\rangle\right]\\
=&-\EE\left[2 \mathrm{Tr}\left[\Cov_{p_t}\Cov_{p_t+q_t}B+\Cov_{q_t,p_t}\left(\Cov_{p_t+q_t}-\Cov_{q_t}\right)B\right]\right]\\
=&-\EE\left[\mathrm{Tr}\left[\left(\Cov_{p_t+q_t}-\Cov_{q_t}\right)^2B\right]\right]+\EE\left[\mathrm{Tr}\left[\Cov_{q_t,p_t}\Cov_{p_t,q_t}B\right]\right]\\
&-\EE\left[\mathrm{Tr}\left[\left(\Cov_{p_t}+\Cov_{q_t,p_t}\right)B\left(\Cov_{p_t}+\Cov_{p_t,q_t}\right)\right]\right]-2\EE\left[\mathrm{Tr}\left[\Cov_{p_t}\Cov_{q_t}B\right]\right]\\
\leq&-\EE\left[\mathrm{Tr}\left[\left(\Cov_{p_t+q_t}-\Cov_{q_t}\right)^2B\right]\right]+\EE\left[\mathrm{Tr}\left[\Cov_{q_t,p_t}\Cov_{p_t,q_t}B\right]\right]-2\EE\left[\mathrm{Tr}\left[\Cov_{p_t}\Cov_{q_t}B\right]\right].
\end{aligned}\,
\end{equation*}
The first term becomes:
\begin{equation}\label{firstterm}
\begin{aligned}
&-\EE\left[\mathrm{Tr}\left[\left(\Cov_{p_t+q_t}-\Cov_{q_t}\right)^2B\right]\right]\leq -\lambda_{\min}(B)\EE\left[\mathrm{Tr}\left[\left(\Cov_{p_t+q_t}-\Cov_{q_t}\right)^2\right]\right]\\
\leq&-\lambda_{\min}(B)\EE\left\|\Cov_{p_t+q_t}-\Cov_{q_t}\right\|^2_F=-\lambda_{\min}(B)\EE\left\|\Cov_{x_t+v_t}-\Cov_{v_t}\right\|^2_F,
\end{aligned}
\end{equation}
while the second term can be bounded by applying~\eqref{dealPQ} in Appendix \ref{sec:twolemmas} Lemma \ref{firstlemma1}:
\begin{equation*}
\begin{aligned}
&\EE\left[\mathrm{Tr}\left[\Cov_{q_t,p_t}\Cov_{p_t,q_t}B\right]\right]\leq \|B\|_2\EE\left[\mathrm{Tr}\left[\Cov_{q_t,p_t}\Cov_{p_t,q_t}\right]\right]\\
=&\|B\|_2\EE\left\|\Cov_{p_t,q_t}\right\|^2_F\leq \|B\|_2\textrm{Var}(\rho(t))\EE|p^1_t|^2+C J^{-1/2-\alpha(1-\epsilon)}\,,
\end{aligned}
\end{equation*}
where the last inequality comes from \eqref{dealPQ} in Appendix \ref{sec:twolemmas} Lemma \ref{firstlemma1}. 

And the third term:
\begin{equation}\label{F3term}
\begin{aligned}
\left|\EE\left[\mathrm{Tr}\left[\Cov_{p_t}\Cov_{q_t}B\right]\right]\right|&\leq \EE\left(\sum^J_{j=1}\left\langle p^j_t,\Cov_{q_t}B p^j_t\right\rangle\right)=\EE\left(\sum^J_{j=1}\left\langle p^j_t,\Cov_{\rho(t)}B p^j_t\right\rangle+\sum^J_{j=1}\left\langle p^j_t,\left(\Cov_{q_t}-\Cov_{\rho(t)}\right)B p^j_t\right\rangle\right)\\
&\leq \|B\|_2\textrm{Var}(\rho(t))\mathbb{E}|p^1|^2+\EE\left(\sum^J_{j=1}\|\left(\Cov_{q_t}-\Cov_{\rho(t)}\right)B\|_2\left|p^1\right|^2\right)\\
&\leq \|B\|_2\textrm{Var}(\rho(t))\mathbb{E}|p^1|^2+C J^{-1/2}\left(\EE\left|p^1\right|^2\right)^{1-\epsilon}\\
&\leq \|B\|_2\textrm{Var}(\rho(t))\mathbb{E}|p^1|^2+C J^{-1/2-\alpha(1-\epsilon)},
\end{aligned}
\end{equation}
where we used the same techniques as in \eqref{Example1} in Appendix~\ref{sec:twolemmas}.

Combine~\eqref{diffR},~\eqref{traceterm}-\eqref{F3term} into~\eqref{xoverlinex2diff}, we finally have:
\[
\begin{aligned}
\frac{d\EE|p^1_t|^2}{dt}\leq& 2\|B\|_2\textrm{Var}(\rho(t))\EE|p^1_t|^2 - \left(\lambda_{\min}(B)-\frac{1}{\lambda_0}\right)\EE\left[\mathrm{Tr}\left[\left(\Cov_{p_t+q_t}-\Cov_{q_t}\right)^2\right]\right]\\
&+ C_\epsilon J^{-1/2-\alpha(1-\epsilon)}+ C_\epsilon J^{-1/2-\alpha/2+\epsilon\alpha/4}\,.
\end{aligned}
\]
Under the assumption that $\lambda_{\min}(B)\lambda_0\geq 1$, and with $\EE|p^1_0|^2=0$, we apply Gr\"onwall inequality for $\EE|p^1_t|^2$ to obtain
\begin{equation*}
\mathbb{E}|p^1_t|^2\leq C_\epsilon J^{-1/2-\alpha/2+\epsilon\alpha/4}\,,
\end{equation*}
for all finite time, finishing the proof for~\eqref{pestimation}.
\end{proof}

\begin{proof}[Proof of Lemma~\ref{lem:bound_x}]
To prove~\eqref{xestimation}, we first note, citing Lemma~\ref{secondlemma1} in Appendix~\ref{sec:twolemmas} and use the result from Lemma~\ref{lem:bound_p}:

\begin{equation*}
\left(\EE\|\Cov_{x_t+v_t}-\Cov_{v_t}\|^2_2\right)^{1/2}\leq C_{\epsilon} J^{-1/4-\alpha/4+\epsilon/2}\,.
\end{equation*}

This helps us to control each term in~\eqref{x2diff}:
\begin{itemize}
\item[1.]
\begin{equation*}
\begin{aligned}
\EE\left\langle x^j_t,\left(\Cov_{x_t+v_t}-\Cov_{v_t}\right)Bv^j_t\right\rangle&\leq \left(\EE|x^j_t|^2|Bv^j_t|^2\right)^{1/2}\left(\EE\|\Cov_{x_t+v_t}-\Cov_{v_t}\|^2_2\right)^{1/2}\\&\leq C_{\epsilon} J^{-1/4-\alpha/4+\epsilon/2}\left(\EE|x^1_t|^2\right)^{(2-\epsilon)/4}
\end{aligned}\,;
\end{equation*}
\item[2.]
\begin{equation*}
\begin{aligned}
\EE\left\langle x^j_t,\left(\Cov_{x_t+v_t}-\Cov_{v_t}\right)Bu^*\right\rangle&\leq \left(\EE|x^j_t|^2|Bu^*|^2\right)^{1/2}\left(\EE\|\Cov_{x_t+v_t}-\Cov_{v_t}\|^2_2\right)^{1/2}\\&\leq C_{\epsilon} J^{-1/4-\alpha/4+\epsilon/2}\left(\EE|x^1_t|^2\right)^{1/2}\\
&\leq C_{\epsilon} J^{-1/4-\alpha/4+\epsilon/2}\left(\EE|x^1_t|^2\right)^{(2-\epsilon)/4}
\end{aligned}\,;
\end{equation*}
\item[3.]
\[
\begin{aligned}
\EE\mathrm{Tr}\left(\sqrt{\Cov_{x_t+v_t}}-\sqrt{\Cov_{\rho(t)}}\right)^2\leq \frac{C_{\epsilon}}{\lambda_0}J^{-1/2-\alpha/2+\epsilon}\leq C_{\epsilon} J^{-1/2-\alpha/2+\epsilon}\,;
\end{aligned}
\]
\item[4.]
\begin{equation*}
\begin{aligned}
\EE R^j_t&\leq \left(\EE|x^j_t|^2|Bv^j_t|^2\right)^{1/2}\left(\EE\|\Cov_{v_t}-\Cov_{\rho(t)}\|^2_2\right)^{1/2}+\left(\EE|x^j_t|^2|Bu^*|^2\right)^{1/2}\left(\EE\|\Cov_{v_t}-\Cov_{\rho(t)}\|^2_2\right)^{1/2}\\
&\leq C_{\epsilon} J^{-1/4-\alpha/4+\epsilon/2}\left(\EE|x^1_t|^2\right)^{(2-\epsilon)/4}\,,
\end{aligned}
\end{equation*}
\end{itemize}
where we use H\"older's inequality and uniform boundedness of high moments, stated in Proposition~\ref{Bforv} \eqref{highmomentv}-\eqref{highmomentv4} and Lemma~\ref{cor:prebound} \eqref{prebound} in these estimations. Now, we rewrite:
\[
\frac{d\mathbb{E}|x^1_t|^2}{dt}\leq C_\epsilon J^{-1/4-\alpha/4+\epsilon/2}\mathbb{E}|x^1_t|^{(2-\epsilon)/4}+J^{-1/2-\alpha/2+\epsilon}\,,
\]
and with $\mathbb{E}|x^1_0|^2=0$, we finally have:
\[
\mathbb{E}|x^1_t|^2\leq C_\epsilon J^{-(1+\alpha-2\epsilon)/(2-\epsilon)},
\]
by the Gr\"onwall inequality. This finishes the proof.
\end{proof}

\section{Conclusion}
In this paper we give the rigorous justification of the validity of EKS algorithm as a sampling method when the forward map is linear. The composition of the proof largely follows the strategy of the classical coupling method. The dynamics of the particles is described by a coupled SDE system, which in the large sampling limit, converges to a Fokker-Planck equation whose long time equilibrium is the target distribution. In the nonlinear setting, we claim the method is bound to be wrong, and we give the argument in Appendix~\ref{sec:appendixnonlinear}. However, this is \emph{not} to say that the method is not useful in the nonlinear setting: the main attraction of the method is that it provides a way to achieves gradient-free property. With rigorous understanding of the algorithm, the studies shown here pave the way for further designing gradient-free algorithms.

\begin{appendix}
\section{Proof of Proposition \ref{Bforv}}\label{ap:proofofBforv}
The proof is similar to \cite{ding2019meanfield} (Lemma 3). The bounds in~\eqref{highmomentv} are immediate considering Lemma \ref{lem3:EnFL}. We only show~\eqref{diffforconv} here. Without loss of generality, assume $\mathbb{E}(v^j_t)=0$, then we write $\Cov_{v_t}$ as
\[
\Cov_{v_t}=\frac{J-1}{J^2}\left(\sum^{J}_{j=1}v^j_t\otimes v^j_t\right)-\frac{1}{J^2}\sum^J_{j\neq k}v^j_t\otimes v^k_t\,.
\]
Now we divide \eqref{diffforconv} into three parts
\[
\begin{aligned}
&\left(\mathbb{E}\left\|\Cov_{v_t}-\Cov_{\rho(t)}\right\|^p_2\right)^{1/p}\\
\leq&\left(\mathbb{E}\left\|\frac{1}{J}\left(\sum^{J}_{j=1}v^j_t\otimes v^j_t\right)-\frac{1}{J}\left(\sum^{J}_{j=1}\Cov_{\rho(t)}\right)\right\|^p_2\right)^{1/p}\\
&\ +\left(\mathbb{E}\left\|\frac{1}{J^{2}}\left(\sum^J_{j=1}v^j_t\right)\otimes\left(\sum^J_{k=1}v^k_t\right)\right\|^p_2\right)^{1/p}\,.
\end{aligned}
\]
To control the first term, we have
\[
\begin{aligned}
&\left(\mathbb{E}\left\|\frac{1}{J}\left(\sum^{J}_{j=1}v^j_t\otimes v^j_t\right)-\frac{1}{J}\left(\sum^{J}_{j=1}\Cov_{\rho(t)}\right)\right\|^p_2\right)^{1/p}\\
\leq&\,C_{p,L}\left(\mathbb{E}\left\|\frac{1}{J}\left(\sum^{J}_{j=1}v^j_t\otimes v^j_t\right)-\Cov_{\rho(t)}\right\|^p_F\right)^{1/p}\\
\leq&\,C_{p,L}\sum^L_{m,n=1} \left(\mathbb{E}\left(\frac{1}{J}\sum^{J}_{j=1}v^j_t\otimes v^j_t-\Cov_{\rho(t)}\right)_{m,n}^{p}\right)^{1/p}\\
=&\,\frac{C_{p,L}}{J^{\frac{1}{2}}}\sum^L_{m,n=1} \left\{\mathbb{E}\left[\frac{\sum^{J}_{j=1}\left(v^j_t\otimes v^j_t-\Cov_{\rho(t)}\right)_{m,n}}{\sqrt{J}}\right]^{p}\right\}^{1/p},
\end{aligned}
\]
where $\left(\frac{1}{J}\sum^{J}_{j=1}v^j_t\otimes v^j_t-\Cov_{\rho(t)}\right)_{m,n}$ means the $(m,n)^{th}$ entry of matrix. For each $m,n$, define a new sequence of random variables $\{w^j_{m,n}\}^{J}_{j=1}$ as 
\begin{equation}\label{wjmn}
w^j_{m,n}=\left(v^j_t\otimes v^j_t-\Cov_{\rho(t)}\right)_{m,n}\,,
\end{equation}
then they are i.i.d  with zero expectation and finite high moments:
\begin{equation}
\EE(w^j_{m,n})=0,\quad \EE|w^j_{m,n}|^p<\infty\,.
\end{equation}
By \cite{johnson1985}, we have
\begin{equation}\label{wjmn2}
\mathbb{E}\left[\sum^J_{j=1}w^j_{m,n}\right]^p\lesssim J^{p/2}\,,
\end{equation}
which implies
\[
\mathbb{E}\left[\frac{\sum^{J}_{j=1}\left(v^j_t\otimes v^j_t-\Cov_{\rho(t)}\right)_{m,n}}{\sqrt{J}}\right]^{p}\sim O(1)\,.
\]
For the second term, we can use similar argument as \eqref{wjmn}-\eqref{wjmn2}:
\[
\begin{aligned}
&\left(\mathbb{E}\left\|\frac{1}{J^{2}}\left(\sum^J_{j=1}v^j_t\right)\otimes\left(\sum^J_{k=1}v^k_t\right)\right\|^p_2\right)^{1/p}\\
=& \left(\frac{1}{J^{2p}}\mathbb{E}\left|\sum^J_{j=1}v^j_t\right|^{2p}\right)^{1/p}\lesssim (J^{-p})^{1/p}=J^{-1}\,.
\end{aligned}
\]
In conclusion, we finally obtain
\[
\left(\mathbb{E}\left\|\Cov_{v_t}-\Cov_{\rho(t)}\right\|^p_2\right)^{1/p}\lesssim J^{-\frac{1}{2}}\,,
\]
which proves \eqref{diffforconv}. The proof for bounding \eqref{highmomentv4} is similar and is omitted from here.

\section{Two Lemmas}\label{sec:twolemmas}
\begin{lem}\label{firstlemma1}
Let $\{u^j_t\}$ solve \eqref{CAmunew} and $\{v^j_t\}$ solve \eqref{eqn:vj} with same initial condition. The coefficient for $v^j_t$ are determined by $\rho$, the solution to~\eqref{FKPK} with the initial data $\rho_0\in\mathcal{C}^2$ that has finite high moments. Let $x^j_t,p^j_t,q^j_t$ be defined as in~\eqref{eqn:def_x}, then for any $0<\epsilon<1/2$ and $T>0$, there exists a constant $C>0$ independent of $J,t$ such that
\begin{equation}\label{dealPP}
\EE\|\Cov_{p_t}\|^2_F\leq C\left(\EE|p^1_t|^2\right)^{(2-\epsilon)/2}\,
\end{equation}
and
\begin{equation}\label{dealPQ}
\EE\|\Cov_{p_t,q_t}\|^2_F=\EE\|\Cov_{q_t,p_t}\|^2_F\leq\textrm{Var}(\rho(t))\EE|p^1_t|^2+C J^{-1/2}\left(\EE|p^1_t|^2\right)^{1-\epsilon}\,
\end{equation}
for any $0\leq t\leq T$.
\end{lem}
\begin{lem}\label{secondlemma1}
Let $\{u^j_t\}$ solve \eqref{CAmunew} and $\{v^j_t\}$ solve \eqref{eqn:vj} with same initial condition. The coefficient for $v^j_t$ are determined by $\rho$, the solution to~\eqref{FKPK} with the initial data $\rho_0\in\mathcal{C}^2$ that has finite high moments. Let $x^j_t,p^j_t,q^j_t$ be defined as in~\eqref{eqn:def_x}, then for any $0<\epsilon<1/2$ and $T>0$, there exists a constant $C>0$ independent of $J,t$ such that
\begin{equation}\label{eqn:lemma2}
\left(\EE\left\|\Cov_{x_t+v_t}-\Cov_{v_t}\right\|^2_F\right)^{1/2}\leq C \left[(\EE|p^1_t|^2)^{1/2}+J^{-1/2}\left(\EE|p^1_t|^2\right)^{1/2-\epsilon/2}+\left(\EE|p^1_t|^2\right)^{(2-\epsilon)/4}\right]
\end{equation}
for any $0\leq t\leq T$.
\end{lem}

\begin{proof}[Proof of Lemma \ref{firstlemma1}]
We perform the estimates one by one.
\begin{itemize}
\item[1.] to estimate $\EE\|\Cov_{p_t}\|^2_F$:
\begin{equation}
\begin{aligned}
&\EE\|\Cov_{p_t}\|^2_F=\EE\left\{\text{Tr}(\Cov_{p_t}\Cov_{p_t})\right\}\leq\frac{1}{J^2}\EE\left\{\sum^J_{j,k=1}|p^j_t|^2|p^k_t|^2\right\}\\
=&\EE\left(\frac{1}{J}\sum^J_{j=1}|p^j_t|^2\right)^2\leq \frac{1}{J}\sum^J_{j=1}\EE\left(|p^j_t|^4\right)=\EE\left(|p^1_t|^4\right)\\
=&\EE\left(|p^1_t|^{2-\epsilon}|p^1_t|^{2+\epsilon}\right)\leq \left(\EE|p^1_t|^2\right)^{(2-\epsilon)/2}\left(\EE|p^1_t|^{(4+2\epsilon)/\epsilon}\right)^{\epsilon/2}\leq C\left(\EE|p^1_t|^2\right)^{(2-\epsilon)/2}\\
\end{aligned}
\end{equation}
where $C$ is a contant independent of $J,t$ and we used the H\"older's inequality and the boundedness for high moments (Proposition~\ref{cor:prebound} \eqref{prebound});
\item[2.] to estimate $\EE\|\Cov_{p_t,q_t}\|^2_F$ and equally $\EE\|\Cov_{q_t,p_t}\|^2_F$, note:
\begin{equation*}
\begin{aligned}
&\EE\|\Cov_{p_t,q_t}\|^2_F=\EE\|\Cov_{q_t,p_t}\|^2_F=\EE\left\{\text{Tr}(\Cov_{p_t,q_t}\Cov_{q_t,p_t})\right\}\\
=&\frac{1}{J^2}\EE\left\{\sum^J_{i,j=1}\left\langle p^i_t,p^j_t\right\rangle\left\langle q^i_t,q^j_t\right\rangle \right\}\leq \frac{1}{J^2}\EE\left\{\sum^J_{i,j=1}|p^i_t||p^j_t||q^i_t||q^j_t|\right\}\\
=&\frac{1}{J^2}\EE\left(\sum^J_{j=1}|p^j_t||q^j_t|\right)^2\leq \EE\left\{\left(\frac{1}{J}\sum^J_{j=1}|p^j_t|^2\right)\left(\frac{1}{J}\sum^J_{j=1}|q^j_t|^2\right)\right\}\\
=&\EE\left\{\left(\frac{1}{J}\sum^J_{j=1}|p^j_t|^2\right)\textrm{Var}(\rho(t))\right\}+
\EE\left\{\left(\frac{1}{J}\sum^J_{j=1}|p^j_t|^2\right)\left(\left(\frac{1}{J}\sum^J_{j=1}|q^j_t|^2\right)-\textrm{Var}(\rho(t))\right)\right\}.
\end{aligned}\,
\end{equation*}
Since
\begin{equation}\label{Example1}
\begin{aligned}
&\EE\left\{\left(\frac{1}{J}\sum^J_{j=1}|p^j_t|^2\right)\left(\left(\frac{1}{J}\sum^J_{j=1}|q^j_t|^2\right)-\textrm{Var}(\rho(t))\right)\right\}\\
=&\EE\left\{\left(\frac{1}{J}\sum^J_{j=1}|p^j_t|^2\right)^{1-\epsilon}\left(\frac{1}{J}\sum^J_{j=1}|p^j_t|^2\right)^{\epsilon}\left(\left(\frac{1}{J}\sum^J_{j=1}|q^j_t|^2\right)-\textrm{Var}(\rho(t))\right)\right\}\\
\leq&\left(\EE\left(\frac{1}{J}\sum^J_{j=1}|p^j_t|^2\right)\right)^{1-\epsilon}
\left(\EE\left(\frac{1}{J}\sum^J_{j=1}|p^j_t|^2\right)\left(\left(\frac{1}{J}\sum^J_{j=1}|q^j_t|^2\right)-\textrm{Var}(\rho(t))\right)^{1/\epsilon}\right)^{\epsilon}\\
\leq&\left(\EE\left(\frac{1}{J}\sum^J_{j=1}|p^j_t|^2\right)\right)^{1-\epsilon}
\left(\EE\left(\frac{1}{J}\sum^J_{j=1}|p^j_t|^4\right)\right)^{\epsilon/2}\left(\EE\left(\left(\frac{1}{J}\sum^J_{j=1}|q^j_t|^2\right)-\textrm{Var}(\rho(t))\right)^{2/\epsilon}\right)^{\epsilon/2}\\
\leq&C J^{-1/2}\left(\EE\left(\frac{1}{J}\sum^J_{j=1}|p^j_t|^2\right)\right)^{1-\epsilon}\,,
\end{aligned}
\end{equation}
where $C$ is a contant independent of $J,t$ and we use the uniform boundedness of high moments, stated in Proposition~\ref{Bforv} \eqref{highmomentv}-\eqref{highmomentv4} and Proposition~\ref{cor:prebound} \eqref{prebound}. Therefore, 
we have
\begin{equation}
\EE\|\Cov_{p_t,q_t}\|^2_F\leq\textrm{Var}(\rho(t))\EE|p^1_t|^2+C J^{-1/2}\left(\EE|p^1_t|^2\right)^{1-\epsilon}\,.
\end{equation}
by symmetry of $p^j_t$.
\end{itemize}
\end{proof}

\begin{proof}[Proof of Lemma \ref{secondlemma1}]
This is a direct result of 
\[
\begin{aligned}
\left(\EE\left\|\Cov_{x_t+v_t}-\Cov_{v_t}\right\|^2_F\right)^{1/2}&=\left(\EE\left\|\Cov_{p_t+q_t}-\Cov_{q_t}\right\|^2_F\right)^{1/2}\\
&\leq\left(\EE\|\Cov_{p_t}\|^2_F\right)^{1/2}+\left(\EE\|\Cov_{p_t,q_t}\|^2_F\right)^{1/2}+\left(\EE\|\Cov_{q_t,p_t}\|^2_F\right)^{1/2}\,.
\end{aligned}
\]
and Lemma \ref{firstlemma1}.
\end{proof}
\section{EKS with nonlinear $\mathcal{G}$}\label{sec:appendixnonlinear}
In this section we discuss the behavior of the algorithm when $\mathcal{G}$ is nonlinear. We are particularly interested in the following kind of nonlinearity. Assume $\mathcal{G}$ is composed of a linear component and a nonlinear component, meaning there is a matrix $A\in\mathcal{L}(\mathbb{R}^L,\mathbb{R}^K)$ such that
\begin{equation}\label{nonlinear}
\mathcal{G}(u)=Au+\Rm(u)\,,
\end{equation}
where the nonlinear component $\Rm(u):\mathbb{R}^L\to\mathbb{R}^K$ is smooth and bounded, satisfying
\begin{equation}\label{eqn:condition_non}
\text{Range}(\Rm)\perp_{\Gamma^{-1}}\text{Range}(A)\,,\quad \left|\Rm(u)\right|+\left|\nabla_u\Rm(u)\right|\leq M\,,
\end{equation}
with some constant $M>0$ in $\mathbb{R}^L$. Here $a\perp_{\Gamma^{-1}}b$ means $a^\top \Gamma^{-1}b=0$ and $a^\top$ is to the take transpose of $a$. This form of $\mathcal{G}$ requires the linear and the nonlinear components of $\mathcal{G}$ being perpendicular in their ranges, and allows easier computation. We are interested in this special kind of nonlinearity is because it allows an easier extension of the techniques used in the main texts of the paper.

Under this assumption, one is looking for sampling from the target distribution:
\begin{equation}\label{eqn:rho_pos_non}
\rho_{\text{pos}}(u)\propto\exp\left(-\Phi_R\right)\,,\quad\text{where}\quad \Phi_R = \frac{1}{2}|y-\mathcal{G}(u)|^2_{\Gamma}+\frac{1}{2}|u-u_0|^2_{\Gamma_0}\,.
\end{equation}
\\
This distribution is the equilibrium of the following Fokker-Planck equation:
\begin{align}\label{meanfieldstep2appendix}
\partial_t\rho&=\nabla\cdot(\rho\Cov_{\rho(t)}\nabla\Phi_R)+\mathrm{Tr}\left(\Cov_{\rho(t)}D^2\rho\right)\nonumber\\
&=\nabla\cdot(\rho\Cov_{\rho(t)}\nabla \mathcal{G}(u)\Gamma^{-1}(\mathcal{G}(u)-y))+\nabla\cdot(\rho\Cov_{\rho(t)}\Gamma^{-1}_0(u-u_0))+\mathrm{Tr}\left(\Cov_{\rho(t)}D^2\rho\right)\,.
\end{align}
\\
This means, a particle should follow the flow, for all $j=1\,,\cdots\,,J$
\begin{align}\label{eqn:mu_3}
du^{j}_t&=-\mathrm{Cov}_{\rho_t}\nabla\Phi_R(u^j_t)dt+\sqrt{2\mathrm{Cov}_{\rho_t}}dW^{j}_t\nonumber\\
& = -\mathrm{Cov}_{\rho_t}\left(\nabla \mathcal{G}(u^{j}_t)\Gamma^{-1}(\mathcal{G}(u^{j}_t)-y)+\Gamma^{-1}_0(u^{j}_t-u_0)\right)dt+\sqrt{2\mathrm{Cov}_{\rho_t}}dW^{j}_t\,.
\end{align}
\\
However, in practice $\rho_t$ is not available, and we replace it by the ensemble distribution, and change the SDE system into, for all $j=1\,,\cdots\,,J$:
\begin{align}\label{eqn:mu_2}
du^{j}_t&=-\mathrm{Cov}_{u_t,u_t}\nabla\Phi_R(u^j_t)dt+\sqrt{2\mathrm{Cov}_{u_t,u_t}}dW^{j}_t\nonumber\\
& = -\mathrm{Cov}_{u_t,u_t}\left(\nabla \mathcal{G}(u^{j}_t)\Gamma^{-1}(\mathcal{G}(u^{j}_t)-y)+\Gamma^{-1}_0(u^{j}_t-u_0)\right)dt+\sqrt{2\mathrm{Cov}_{u_t,u_t}}dW^{j}_t\,.
\end{align}

We have two comments:
\begin{itemize}
\item The SDE system~\eqref{eqn:mu_2} is wellposed. We will prove this in Theorem~\ref{ThforMappendix}.
\item The mean-field limit of the SDE system~\eqref{eqn:mu_2} is the Fokker-Planck equation~\ref{meanfieldstep2appendix}. To show this amounts to perform the coupling method and compare the two systems~\eqref{eqn:mu_3} and~\eqref{eqn:mu_2}. This is formally correct but we do not intend to prove it in this paper. The techniques can be found in~\cite{ding2019meanfield} where the authors show the validity of Ensemble Kalman Inversion \cite{Iglesias_2013} using the same argument with the same kind of nonlinearity.
\end{itemize}
This discussion suggests that the ensemble distribution of~\eqref{eqn:mu_2}, in the mean-field limit, converges to the solution to~\eqref{meanfieldstep2appendix}, whose long time limit is the target distribution. This means that~\eqref{eqn:mu_2}, instead of~\eqref{Ecov_dis_revise}, the SDE system used in EKS, is the correct flow for capturing $\rho_{\text{pos}}(u)$. Comparing the two flows, it is straightforward to see that these two flows are different merely because of the simple fact that
\[
\Cov_{\rho(t)}\nabla\mathcal{G}(u)\neq \Cov_{\rho(t),\mathcal{G}}
\]
for almost all nonlinear $\mathcal{G}$.

This concludes that EKS is a sampling method that does \emph{not} capture the target distribution if the forward map is nonlinear, but a small modification by using~\eqref{eqn:mu_2} can make it consistent. We summarize it into Algorithm~\ref{ALG2}.

\begin{algorithm}[h]
\caption{\textbf{Ensemble Kalman sampler, nonlinear}}\label{ALG2}
\begin{algorithmic}
\State \textbf{Preparation:}
\State 1. Input: $J$ (number of particles); $h$ (stepsize); $N$ (stopping index); $\Gamma$; $\Gamma_0$; and $y$ (data).
\State 2. Initial: $\{u^j_0\}$ sampled from a initial distribution induced by a density function $\rho_0$.
\State \textbf{Run: } Set time step $n=0$;
\State \textbf{While} $n<N$:
1. Define empirical means and covariance:
\begin{align}\label{eqn:en_mean_var_2}
\overline{u}_n=\frac{1}{J}\sum^J_{j=1}u^j_n\,,\quad&\text{and}\quad\overline{\MCG}_n=\frac{1}{J}\sum^J_{j=1}\MCG(u^j_n)\,,\nonumber\\
\Cov_{u_n,u_n}=\frac{1}{J}\sum^J_{j=1}&\left(u^j_n-\overline{u}_n\right)\otimes \left(u^j_n-\overline{u}_n\right)\,.
\end{align}
2. Update ensemble particles ($\forall 1\leq j\leq J$)
\begin{equation}\label{eqn:update_ujn_2}
\begin{aligned}
&u^j_{*,n+1}=u^j_n-h\Cov_{u_n,u_n}\nabla \mathcal{G}(u^{j}_n)\Gamma^{-1}\left(\MCG(u^j_n)-y\right)-h\Cov_{u_n,u_n}\Gamma^{-1}_0\left(u^j_{*,n+1}-u_0\right)\,,\\
&u^j_{n+1}=u^j_{*,n+1}+\sqrt{2h\Cov_{u_n,u_n}}\xi^j_n\,,\quad\text{with}\quad\xi^j_{n+1}\sim \mathcal{N}(0,\mathrm{I})\,.
\end{aligned}
\end{equation}
3. Set $n\to n+1$.
\State \textbf{end}
\State \textbf{Output:} Ensemble particles $\{u^j_N\}$.
\end{algorithmic}
\end{algorithm}

We should emphasize that one key disadvantage of this algorithm, compared to Algorithm~\ref{ALG1} is that it loses the ``gradient-free" property. Despite EKS fails to be consistent in the nonlinear regime, it serves as a stepping stone for constructing new algorithms that both enjoy the gradient-free properties and are consistent.

Now we state our wellposedness theory:
\begin{theorem}\label{ThforMappendix}
Suppose $\mathcal{G}$ satisfies \eqref{nonlinear}, if $\left\{u^j_0:\Omega\rightarrow\mathcal{X}\right\}^J_{j=1}$ is independent almost surely, then for all $t\geq0$, there exists a unique strong solution $(u^j_t)^J_{j=1}$ (up to $\mathbb{P}$-indistinguishability) of the set of coupled SDEs \eqref{CAmunew} (equivalent to~\eqref{eqn:mu_2}), where the cost function $\Phi_R$ is defined in~\eqref{eqn:rho_pos_non}.
\end{theorem}


\begin{proof} First, considering each $u^j$ is a vector of $L$-length, \eqref{eqn:mu_2} can be written as:
\[
dU_t=(F_1(U_t)+F_2(U_t))dt+G(U_t)dW_t\,,
\]
where $U_t=\left(u^j_t\right)^J_{j=1}\in\RR^{LJ\times1}$, $W_t=\left(W^j_t\right)^J_{j=1}\in\RR^{LJ\times1}$ and
\begin{align*}
&F_1(U_t)=\left(-\mathrm{Cov}_{u_t}B\left(u^j_t-u^\ast\right)\right)^J_{j=1}\in\RR^{LJ\times1}\,,\\
&F_2(U_t)=\left(-\mathrm{Cov}_{u_t}\nabla \Rm(u^j_t)\Gamma^{-1}\left(\Rm(u^j_t)-r\right)\right)^J_{j=1}\in\RR^{LJ\times1}\,,\\
&G(U_t)=\text{diag}\left(\sqrt{2\mathrm{Cov}_{u_t}}\right)^J_{j=1}\,,\\
\end{align*}
where $\mathrm{Cov}_{u}$ is the empirical covariance and $\text{diag}(D_j)^J_{j=1}$ is a diagonal block matrix with matrices $\left(D_j\right)^J_{j=1}$ on the diagonal and $B$ and $u^\ast$ are defined in~\eqref{def:Bustar}.
\\
Use the same Lynapunov funtion as \eqref{eqn:3.2V}, according to stochastic Lyapunov theory, it suffices to prove
\begin{equation}\label{eqn:3.2appendix1}
\begin{aligned}
LV(U):=\nabla V(U)\cdot (F_1(U)+F_2(U))+\frac{1}{2}\text{Tr}\left(G^\top(U)\text{Hess}[V](U)G(U)\right)\leq cV(U)\,.
\end{aligned}
\end{equation}
\\
Use \eqref{eqn:3.2V1},\eqref{eqn:3.2V2}, we have
\begin{equation}\label{eqn:3.2appendix2}
\begin{aligned}
&\nabla V(U)\cdot F_1(U)+\frac{1}{2}\text{Tr}\left(G^\top(U)\text{Hess}[V](U)G(U)\right)\\
\leq &-2\mathrm{Tr}\left(\mathrm{Cov}_{u}B\mathrm{Cov}_{u}\right)-2\left\langle B\left(\bar{u}-u^*\right), \mathrm{Cov}_{u}B(\bar{u}-u^*)\right\rangle+2(1+\|B\|_2)V(U)\\
\leq &-2\lambda_{\min}(B)\|\mathrm{Cov}_{u}\|^2_F-2\left\langle B\left(\bar{u}-u^*\right), \mathrm{Cov}_{u}B(\bar{u}-u^*)\right\rangle+2(1+\|B\|_2)V(U)
\end{aligned}
\end{equation}
Therefore, we only need to deal with $\nabla V(U)\cdot F_2(U)$. Use the boundedness of $\Rm$, by~\eqref{eqn:condition_non}, we have
\begin{align*}
\nabla V_1(U)\cdot F_2(U)=&-\frac{2}{J}\sum^J_{j=1}\left\langle u^j-\bar{u}, \mathrm{Cov}_{u}\nabla \Rm(u^j)\Gamma^{-1}(\Rm(u^j)-r)\right\rangle\\
=&-\frac{2}{J}\sum^J_{j=1}\left\langle \sqrt{\mathrm{Cov}_{u}}(u^j-\bar{u}), \sqrt{\mathrm{Cov}_{u}}\nabla \Rm(u^j)\Gamma^{-1}(\Rm(u^j)-r)\right\rangle\,,\\
\leq &\frac{2\lambda_{\min}(B)}{J}\sum^J_{j=1}\left\langle \sqrt{\mathrm{Cov}_{u}}(u^j-\bar{u}),\sqrt{\mathrm{Cov}_{u}}(u^j-\bar{u})\right\rangle\\
&+\frac{1}{2J\lambda_{\min}(B)}\sum^J_{j=1}\left\langle \sqrt{\mathrm{Cov}_{u}}\nabla \Rm(u^j)\Gamma^{-1}(\Rm(u^j)-r),\sqrt{\mathrm{Cov}_{u}}\nabla \Rm(u^j)\Gamma^{-1}(\Rm(u^j)-r)\right\rangle\\
\leq &2\lambda_{\min}(B)\|\Cov_{u}\|^2_F+\frac{M^4\|\Gamma^{-1}\|^2_2}{2\lambda_{\min}(B)}\|\Cov_{u}\|_2\,,
\end{align*}
where we use Young's inequality in the second last inequality. Similarly, we also have
\begin{align*}
\nabla V_2(U)\cdot F_2(U)&=-\frac{2}{J}\sum^J_{j=1}\left\langle B\left(\bar{u}-u^*\right), \mathrm{Cov}_{u}\nabla \Rm(u^j)\Gamma^{-1}(\Rm(u^j)-r)\right\rangle\\
&=-\frac{2}{J}\sum^J_{j=1}\left\langle \sqrt{\mathrm{Cov}_{u}}B\left(\bar{u}-u^*\right), \sqrt{\mathrm{Cov}_{u}}\nabla \Rm(u^j)\Gamma^{-1}(\Rm(u^j)-r)\right\rangle\\
&\leq 2\left\langle B\left(\bar{u}-u^*\right), \mathrm{Cov}_{u}B(\bar{u}-u^*)\right\rangle+\frac{M^4\|\Gamma^{-1}\|^2_2}{2}\|\Cov_{u}\|_2\,,
\end{align*}
where we use Young's inequality in the last inequality.
\\
Combine with \eqref{eqn:3.2appendix2}, use $\|\Cov_{u}\|_2\leq V_1(U)$, we have
\[
\nabla V(U)\cdot (F_1(U)+F_2(U))+\frac{1}{2}\text{Tr}\left(G^\top(U)\text{Hess}[V](U)G(U)\right)\leq \left[\frac{M^4\|\Gamma^{-1}\|^2_2}{2}+\frac{M^4\|\Gamma^{-1}\|^2_2}{2\lambda_{\min}(B)}+2(1+\|B\|_2)\right]V(U)\,,
\]
which proves \eqref{eqn:3.2appendix1} and thus the theorem.
\end{proof}
\end{appendix}
\bibliographystyle{plain}
\bibliography{enkf}

\end{document}